\documentclass[10pt,a4paper]{article} 
\usepackage{color}
\usepackage[utf8]{inputenc}

\usepackage{amsmath, amssymb, amsthm, latexsym}
\usepackage[english]{babel}

\usepackage{float}
\usepackage{hyperref}

\usepackage{enumerate}
\usepackage{mathrsfs}
\usepackage{eucal}
\usepackage{graphicx}

\newcommand{\sca}[2]{\langle#1,#2\rangle}
\newcommand{\barr}{\overline}

\newcommand{\beq}{\begin{equation}}
\newcommand{\eeq}{\end{equation}}

\newcommand{\Supp}{\textrm{Supp~}}

\renewcommand{\Re}{{\rm Re\,}}
\renewcommand{\Im}{{\rm Im\,}}
\newtheorem{theorem}{Theorem}[section]
\newtheorem{lemma}[theorem]{Lemma}

\newtheorem{proposition}[theorem]{Proposition}

\theoremstyle{definition}
\newtheorem{Remark}[theorem]{Remark}

\newtheorem{Example}{Example}

\newcommand{\ie}{\emph{i.e.}}
\newcommand{\eg}{\emph{e.g.}}
\newcommand{\cf}{\emph{cf.}}
\newcommand{\eps}{\varepsilon}
\newcommand{\Com}{\mathbb{C}}
\newcommand{\Real}{\mathbb{R}}
\newcommand{\sgn}{\mathop{\mathrm{sgn}}\nolimits}
\newcommand{\dist}{\mathop{\mathrm{dist}}\nolimits}
\newcommand{\Dom}{\mathsf{Dom}}
\newcommand{\Ran}{\mathsf{Ran}}
\newcommand{\Num}{\mathsf{Num}}
\newcommand{\sii}{L^2}
\newcommand{\CS}{\mathcal{S}}

\numberwithin{equation}{section}

\usepackage[normalem]{ulem}
\usepackage{color}
\definecolor{DarkBlue}{rgb}{0,0.1,0.7}

\newcommand\soutD{\bgroup\markoverwith
{\textcolor{DarkBlue}{\rule[.01ex]{2pt}{1pt}}}\ULon}
\newcommand{\Hm}[1]{\leavevmode{\marginpar{\tiny%
$\hbox to 0mm{\hspace*{-0.5mm}$\leftarrow$\hss}%
\vcenter{\vrule depth 0.1mm height 0.1mm width \the\marginparwidth}%
\hbox to
0mm{\hss$\rightarrow$\hspace*{-0.5mm}}$\\\relax\raggedright #1}}}

\begin{document}

\title{\textbf{Pseudospectra of the Schr\"odinger operator
with a discontinuous complex potential}}

\author{Rapha\"el Henry$^{a}$ \ and \ David Krej\v{c}i\v{r}\'ik$^{b}$}

\date{\small
\emph{
\begin{quote}
\begin{itemize}
\item[$a)$] 
D\'epartement de Math\'ematiques, 
Universit\'e Paris-Sud, B\^at.~425, 
91405 Orsay Cedex, France;
raphael.henry@math.u-psud.fr.%
\\
\item[$b)$]
Department of Theoretical Physics, Nuclear Physics Institute ASCR,
25068 \v{R}e\v{z}, Czech Republic;
krejcirik@ujf.cas.cz.%
\end{itemize}
\end{quote}
}
\medskip
29 September 2015}

\maketitle

\begin{abstract}
\noindent
We study spectral properties of the Schr\"odinger
operator with an imaginary sign potential on the real line.
By constructing the resolvent kernel, 
we show that the pseudospectra of this operator are highly non-trivial, 
because of a blow-up of the resolvent at infinity.
Furthermore, we derive estimates on the location of eigenvalues
of the operator perturbed by complex potentials.
The overall analysis demonstrates striking differences 
with respect to the weak-coupling behaviour of the Laplacian.
\bigskip
\begin{itemize}
\item[\textbf{Keywords:}] 
pseudospectra, non-self-adjointness, 
Schrödinger operators, discontinuous potential,
weak coupling, Birman-Schwinger principle

\item[\textbf{MSC (2010):}] 34L15, 47A10, 47B44, 81Q12

\end{itemize}
\end{abstract}

\newpage
\section{Introduction}
%
Extensive work has been done recently in understanding 
the spectral properties of non-self-adjoint operators
through the concept of \emph{pseudospectrum}.
Referring to by now classical monographs by 
Trefethen and Embree~\cite{Trefethen-Embree} and Davies~\cite{Davies_2007}, 
we define the pseudospectrum of an operator~$T$
in a Hilbert space~$\mathcal{H}$
to be the collection of sets
\begin{equation}\label{pseudospectrum}
  \sigma_\varepsilon(T) 
  :=   
  \sigma(T) \cup
  \left\{ z \in \Com : \, \|(T-z)^{-1}\|> \varepsilon^{-1}
  \right\}
  \,,
\end{equation}
parametrised by $\eps > 0$,
where $\|\cdot\|$ is the operator norm of~$\mathcal{H}$. 
If~$T$ is self-adjoint (or more generally normal), 
then~$\sigma_\varepsilon(T)$ is just an $\eps$-tubular neighbourhood
of the spectrum $\sigma(T)$.
Universally, however, the pseudospectrum is a much more reliable 
spectral description of~$T$ than the spectrum itself.
For instance, it is the pseudospectrum that measures 
the instability of the spectrum under small perturbations 
by virtue of the formula
\begin{equation}\label{instability}
  \sigma_\varepsilon(T) = \bigcup_{\|U\|\leq 1}\sigma(T+\varepsilon U)\,.
\end{equation}

Leaving aside a lot of other interesting situations,
let us recall the recent results when~$T$ is a differential operator.
As a starting point we take the harmonic-oscillator Hamiltonian 
with complex frequency, which is also known as the rotated 
or Davies' oscillator
(see~\cite[Sec.~14.5]{Davies_2007} for a review and references).
Although the complexification has a little effect on the spectrum
(the eigenvalues are just rotated in the complex plane),
a careful spectral analysis reveals drastic changes 
in basis and other more delicate spectral properties of the operator,
in particular, the spectrum is highly unstable against small perturbations,
as a consequence of the pseudospectrum containing 
regions very far from the spectrum. 
Similar peculiar spectral properties have been established 
for complex anharmonic oscillators 
(to the references quoted in~\cite[Sec.~14.5]{Davies_2007},
we add \cite{Henry,Mityagin-2013} for the most recent results),
quadratic elliptic operators 
\cite{Pravda-Starov_2008,Hitrik-Sjostrand-Viola_2013,Viola_2013},
complex cubic oscillators 
\cite{SK,Henry_2014,KSTV,Novak_2015},
and other models 
(see the recent survey~\cite{KSTV} and references therein). 

A distinctive property of the complexified harmonic oscillator
is that the associated spectral problem is explicitly 
solvable in terms of special functions.
A powerful tool to study the pseudospectrum in the situations
where explicit solutions are not available 
is provided by microlocal analysis
\cite{Davies_1999-NSA,Zworski_2001,Dencker-Sjostrand-Zworski_2004}.
The weak point of the semiclassical methods is the usual
hypothesis that the coefficients of the differential operator 
are smooth enough
(\eg~the potential of the Schr\"odinger operator must be at least continuous),
and it is indeed the case of all the models above.
Another common feature of the differential operators 
whose pseudospectrum has been analysed so far 
is that their spectrum consists of discrete eigenvalues only.

The objective of the present work is to enter an unexplored area
of the pseudospectral world by studying the pseudospectrum
of a non-self-adjoint Schr\"odinger operator whose 
\emph{potential is discontinuous} 
and, at the same time, such that the \emph{essential spectrum is not empty}.
Among various results described below, 
we prove that the pseudospectrum is non-trivial,
despite the boundedness of the potential.
Namely, we show that the norm of the resolvent
can become arbitrarily large outside a fixed neighbourhood of its spectrum. 
We hope that our results will stimulate further analysis
of non-self-adjoint differential operators with singular coefficients. 

\section{Main results}\label{Sec.results}
%
In this section we introduce our model 
and collect the main results of the paper.
The rest of the paper is primarily devoted to proofs,
but additional results can be found there, too.

\subsection{The model}
Motivated by the role of step-like potentials as toy models in quantum mechanics, 
in this paper we consider the Schr\"odinger operator in $\sii(\Real)$
defined by
\begin{equation}\label{operator}
  H := -\frac{d^2}{dx^2} + i\,\sgn(x)
  \,, \qquad
  \Dom(H) := W^{2,2}(\mathbb{R})
  \,.
\end{equation}
In fact, $H$~can be considered as an infinite version 
of the $\mathcal{PT}$-symmetric square well introduced in~\cite{Znojil_2001}
and further investigated in \cite{Znojil-Levai_2001,Siegl-2011-50}.

Note that~$H$ is obtained as a bounded perturbation 
of the (self-adjoint) Hamiltonian of a free particle in quantum mechanics,
which we shall simply denote here by $-\Delta$.
Consequently, $H$~is well defined (\ie\ closed and densely defined).
In fact, $H$~is m-sectorial with the numerical range
(defined, as usual, by the set of all complex numbers $(\psi,H\psi)$
such that $\psi \in \Dom(H)$ and $\|\psi\|=1$)
coinciding with the closed half-strip
\begin{equation}\label{Num}
  \Num(H) = \overline{\CS}
  \,, \qquad \mbox{where} \qquad
  \CS := [0,+\infty)+i\,(-1,1)
  \,.
\end{equation}

The adjoint of~$H$, denoted here by~$H^*$, 
is simply obtained by changing~$+i$ to~$-i$ in~\eqref{operator}.
Consequently, $H$~is neither self-adjoint nor normal. 
However, it is $\mathcal{T}$-self-adjoint 
(\ie\ $H^*=\mathcal{T}H\mathcal{T}$), 
where~$\mathcal{T}$ is the antilinear operator of complex conjugation
(\ie\ $\mathcal{T}\psi:=\overline{\psi}$).
At the same time, $H$~is $\mathcal{P}$-self-adjoint, 
where~$\mathcal{P}$ is the parity operator defined by
$(\mathcal{P}\psi)(x):=\psi(-x)$.
Finally, $H$~is $\mathcal{PT}$-symmetric in the sense
of the validity of the commutation relation
$[H,\mathcal{PT}]=0$.

Due to the analogy of the time-dependent Schr\"odinger equation 
for a quantum particle subject to an external electromagnetic field
and the paraxial approximation for a monochromatic light propagation 
in optical media \cite{Longhi_2009}, 
the dynamics generated by~\eqref{operator} 
can experimentally be realised using optical systems.
The physical significance of $\mathcal{PT}$-symmetry is 
a balance between gain and loss~\cite{Brody-Graefe_2012}.

\subsection{The spectrum}
As a consequence of~\eqref{Num}, 
the spectrum of~$H$ is contained in~$\overline{\CS}$.
Moreover, the $\mathcal{PT}$-symmetry implies that
the spectrum is symmetric with respect to the real axis.
By constructing the resolvent of~$H$
and employing suitable singular sequences for~$H$,
we shall establish the following result.
\begin{proposition}\label{propSpectrum}
We have
\beq\label{spectrum}
  \sigma(H) = \sigma_{\mathrm{ess}}(H) 
  = [0,+\infty) + i \, \{-1,+1\}\,.
\eeq
\end{proposition}

The fact that the two rays $[0,+\infty) \pm i$ 
form the essential spectrum of~$H$ is expectable, 
because they coincide with the spectrum of 
the shifted Laplacian $-\Delta \pm i$ in $\sii(\Real)$
and the essential spectrum of differential operators
is known to depend on the behaviour of their coefficients 
at infinity only
(\cf~\cite[Sec.~X]{Edmunds-Evans}).
The absence of spectrum outside the rays is less obvious.

In fact, the spectrum in~\eqref{spectrum} is purely continuous,
\ie~$\sigma(H)=\sigma_\mathrm{c}(H)$,
for it can be easily checked that no point from the set
on the right hand side of~\eqref{spectrum} can be an eigenvalue
of~$H$ (as well as~$H^*$). An alternative way how to \emph{a priori} 
show the absence of the residual spectrum of~$H$,
$\sigma_\mathrm{r}(H)$,
is to employ
the $\mathcal{T}$-self-adjointness of~$H$ (\cf~\cite[Sec.~5.2.5.4]{KS-book}). 

\subsection{The pseudospectrum}
Before stating the main results of this paper, 
let us recall that a closed operator~$T$ is said to have 
\emph{trivial pseudospectra} if,
for some positive constant $\kappa$, we have
\[
  \forall \varepsilon>0
  \,, \qquad
  \sigma_\varepsilon(T) \subset 
  \left\{z : \, \dist\big(z,\sigma(T)\big)\leq\kappa\,\varepsilon
  \right\}
  \,,
\]
or equivalently,
\beq\label{estTriv}
  \forall z\in \Com \setminus \sigma(T)
  \,, \qquad
  \|(T-z)^{-1}\| \leq \frac{\kappa}{\dist\big(z,\sigma(T)\big)}
  \,.
\eeq
Normal operators have trivial pseudospectra,
because for them the equality holds in~\eqref{estTriv} with $\kappa=1$.

In view of~\eqref{Num},
in our case \eqref{estTriv}~holds with $\kappa=1$  
if the resolvent set is replaced by $\Com \setminus \overline{\CS}$. 
However, the following statement implies that~\eqref{estTriv} 
cannot hold inside the half-strip~$\CS$.
\begin{theorem}\label{thmEstResIntro}
For all $\varepsilon>0$, there exists a positive constant~$r_0$ such that, 
for all $z\in\CS$ with $\Re z \geq r_0$,
\beq\label{EstResIntro}
  (1-\varepsilon)\,\frac{\Re z}{\sqrt{1-(\Im z)^2}}
  \leq \|(H-z)^{-1}\|\leq 
  4\,(1+\varepsilon)\,\frac{\Re z}{1-|\Im z|}\,.
\eeq
\end{theorem}

Although the estimates give a rather good description of 
the qualitative shape of the pseudospectra,
the constants and dependence on 
$\dist(z,\sigma(H)) = 1-|\Im z|$ for $z\in\CS$
are presumably not sharp.

In view of Theorem~\ref{thmEstResIntro}, 
$H$~represents another example of a $\mathcal{PT}$-symmetric operator
with non-trivial pseudospectra.
The present study can be thus considered as a natural continuation
of the recent works \cite{SK,Henry_2014,KSTV}.
However, let us stress that the complex perturbation 
in the present model is bounded.
Moreover, comparing the present setting with
the situation when~\eqref{operator}
is subject to an extra Dirichlet condition at zero 
(\cf~Section~\ref{Sec.Dirichlet}),
the difference between these two realisations 
is indeed seen on the pseudospectral level only.

Even though the step-like shape of the potential in~\eqref{operator}
is a feature of the present study, we stress that 
the discontinuity by itself is not the source of 
the non-trivial pseudospectra,
see Remark~\ref{RemSingularity} below.

The pseudospectrum of~$H$ computed numerically using
Eigtool~\cite{Wri02a} by Mark Embree
is presented in Figure~\ref{fig.Embree}.

\begin{figure}[h]\begin{center}
\includegraphics[width=1\textwidth]{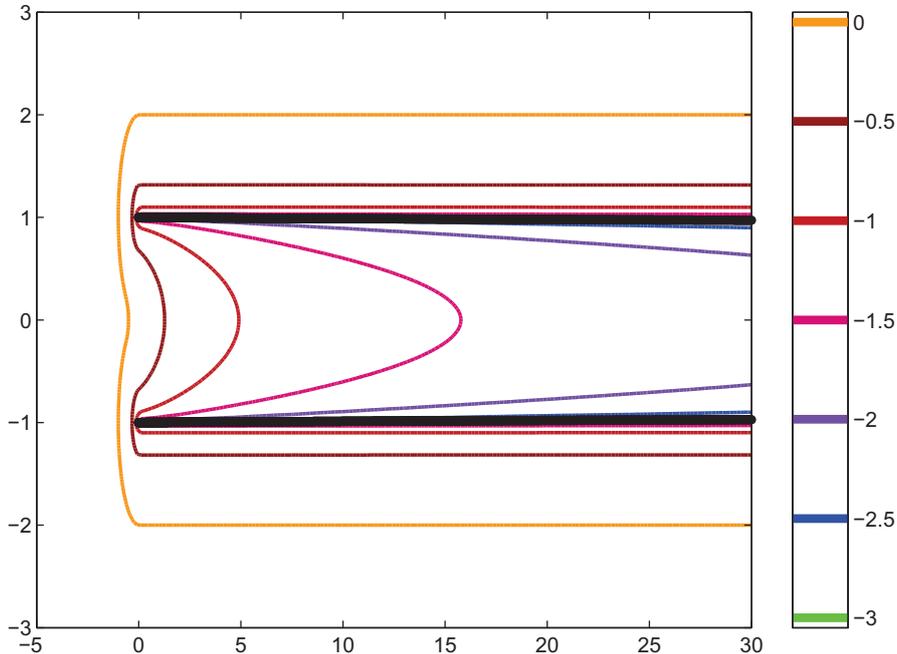}
\caption{The curves $\|(H-z)^{-1}\|=\eps^{-1}$ in the complex
$z$-plane computed for several values of~$\eps$;
the different colours correspond to $\log_{10}\eps$,
while the thick black lines are the essential spectrum of~$H$.
\emph{(Courtesy of Mark Embree.)}}\label{fig.Embree}
\end{center}
\end{figure}

\subsection{Weak coupling}
Inspired by~\eqref{instability}, 
we eventually consider the perturbed operator
\begin{equation}\label{operator.weak}
  H_\varepsilon := H \,\dot{+}\, \varepsilon V
\end{equation}
in the limit as $\eps \to 0$.
Here~$V$ is the operator of multiplication 
by a function $V\in L^1(\mathbb{R})$
that we denote by the same letter.
Since~$V$ is not necessarily relatively bounded with respect to~$H$,
the dotted sum in~\eqref{operator.weak} is understood
in the sense of forms.
We remark that the perturbation does not change
the essential spectrum, \ie,
$
  \sigma_\mathrm{ess}(H_\eps) = \sigma_\mathrm{ess}(H)
$,
and recall Proposition~\ref{propSpectrum}.

If~$H$ were the free Hamiltonian~$-\Delta$ 
and~$V$ were real-valued, the problem~\eqref{operator.weak} with $\eps \to 0$
is known as the regime of \emph{weak coupling} in quantum mechanics. 
In that case, it is well known 
that (under some extra assumptions on~$V$)
the perturbed operator
$-\Delta \,\dot{+}\, \varepsilon V$
possesses a unique discrete eigenvalue for all small positive~$\eps$ 
if, and only if, the integral of~$V$ is non-positive
(see~\cite{Si} for the original work).
This robust existence of ``weakly coupled bound states''
is of course related to the singularity of the resolvent kernel
of the free Hamiltonian at the bottom of the essential spectrum.
Indeed, these bound states do not exist in three and higher dimensions,
which is in turn related to the validity of the Hardy inequality
for the free Hamiltonian (see, \eg, \cite{Weidl_1999a}).

Complex-valued perturbations of the free Hamiltonian
have been intensively studied in recent years
\cite{Abramov-Aslanyan-Davies_2001,Frank-Laptev-Lieb-Seiringer_2006,
Bruneau-Ouhabaz_2008,Laptev-Safronov_2009,Demuth-Hansmann-Katriel_2009,Frank_2011,
Demuth-Hansmann-Katriel_2013}.
In \cite{BK,Novak} the authors consider perturbations of an operator
which is by itself non-self-adjoint.
In all of these papers, however, 
the results are inherited from properties
of the resolvent of the free Hamiltonian.

In the present setting, the unperturbed operator~$H$ is non-self-adjoint.
Moreover, its resolvent kernel has no local singularity,
but it blows up as $|z|\rightarrow+\infty$ when $|\Im z|<1$,
see Section~\ref{sResSpec}.
Consequently, discrete eigenvalues of~$H_\eps$ 
can only ``emerge from the infinity'', 
but not from any finite point of~\eqref{spectrum}.
The statement is made precise by virtue of the following result.
\begin{theorem}\label{thmSpecBound1}
Let $V\in L^1\big(\mathbb{R},(1+x^2)\,dx\big)$.
There exists a positive constant~$C$ (independent of~$V$ and~$\eps$)
such that, whenever
\[
 \varepsilon \, \big\|(1+|\cdot|^2)V\|_{L^1(\mathbb{R})}\leq \frac{1}{C}
  \,,
\]
we have
\beq\label{SpecBound1}
  \sigma_\mathrm{p}(H_\eps)
  \subset
  \overline{\CS}\cap\left\{\Re z\geq\frac{C}{\,\eps^2\,\|V\|_{L^1(\Real)}^2}\right\}
  \,.
\eeq
\end{theorem}

It is interesting to compare this estimate on the location 
of possible eigenvalues of~$H_\eps$
with the celebrated result of~\cite{Abramov-Aslanyan-Davies_2001}
\beq\label{DaviesBound}
  \sigma_\mathrm{p}(-\Delta \,\dot{+}\, \varepsilon V)
  \subset
  \left\{|z|\leq \frac{\,\eps^2\,\|V\|_{L^1(\Real)}^2}{4}\right\}
  \,.
\eeq
Our bound~\eqref{SpecBound1} can be indeed read
as an inverse of~\eqref{DaviesBound}.
It demonstrates how much the present situation differs
from the study of weakly coupled eigenvalues of the free Hamiltonian.

Under some additional assumptions on~$V$,
the claim of Theorem~\ref{thmSpecBound1} can be improved 
in the following way.
\begin{theorem}\label{thmSpecBound2}
Let $n\geq2$ and 
$
  V\in L^1\big(\mathbb{R},(1+x^{2n})\,dx\big)
  \cap W^{1,1}(\mathbb{R})
$.
There exist positive constants~$\varepsilon_0$ and~$C$ such that, 
for all $\varepsilon\in(0,\varepsilon_0)$,
we have
\beq\label{SpecBound2}
  \sigma_\mathrm{p}(H_\eps)
  \subset\overline{\CS}\cap\left\{\Re z \geq 
  \frac{C}{\,\varepsilon^{2n}\,}\right\}
  \,.
\eeq
\end{theorem}

In particular, if for instance $V$ belongs to 
the Schwartz space $\mathscr{S}(\mathbb{R})$,
then every eigenvalue $\lambda(\varepsilon)$ of~$H_\eps$
must ``escape to infinity'' faster than any power of~$\eps^{-1}$ 
as $\eps \to 0$, namely
$
  |\lambda(\varepsilon)|^{-1} = \mathcal{O}(\varepsilon^\infty)
$.

\begin{Remark}
The reader will notice that statement~\eqref{SpecBound1}
differs from~\eqref{SpecBound2} in that the latter
does not highlight the dependence of the right hand side
on the potential~$V$ but only on its amplitude~$\varepsilon$.
The reason is that it is the behaviour of~$H_\eps$ 
on diminishing~$\varepsilon$ that primarily interests us.
Moreover, the proofs of the theorems are different 
and it would be cumbersome (but doable in principle)
to gather the dependence of the right hand side in~\eqref{SpecBound2} 
on (different) norms of~$V$.
\end{Remark}

\subsection{The content of the paper}
The organisation of this paper is as follows.

In Section~\ref{sResSpec}, 
we find the integral kernel of the resolvent $(H-z)^{-1}$,
\cf~Proposition~\ref{propIntegral},
and use it to prove Proposition~\ref{propSpectrum}. 

In Section~\ref{sPseudo},
the explicit formula of the resolvent kernel is further exploited 
in order to prove Theorem~\ref{thmEstResIntro}.

The definition of the perturbed operator~\eqref{operator.weak} 
and its general properties are established in Section~\ref{sPert}.
In particular, we locate its essential spectrum 
(Proposition~\ref{propStabEss})
and prove the Birman-Schwinger principle 
(Theorem~\ref{thmBirmanSchwinger}).

Section~\ref{sBS} is divided into two respective subsections,
in which we prove Theorems~\ref{thmSpecBound1} and~\ref{thmSpecBound2}
with help of the Birman-Schwinger principle
and, again, using the explicit formula of the resolvent kernel.

Finally, in Section~\ref{Sec.Ex}, we present two concrete examples
of the perturbed operator~\eqref{operator.weak}.
Moreover, we make a comparison of the present study
with a decoupled model due to an extra Dirichlet condition.

\section{The resolvent and spectrum}\label{sResSpec}
%
Our goal in this section is to obtain 
an integral representation of the resolvent of~$H$.
Using that result, we give a proof of Proposition~\ref{propSpectrum}.

In the following, we set 
$$
  k_+(z) := \sqrt{i-z} 
  \qquad\mbox{and}\qquad 
  k_-(z) := \sqrt{-i-z}
  \,,
$$ 
where we choose the principal value of the square root,
\ie, $z \mapsto \sqrt{z}$
is holomorphic on $\mathbb{C}\setminus(-\infty,0]$ and positive on $(0,+\infty)$.

\begin{proposition}\label{propIntegral}
For all $z\notin\mathbb{R}_++i\,\{-1,1\}\,$, 
$H-z$ is invertible and, for every $f\in L^2(\mathbb{R})\,$,
\beq\label{integralResolvent}
  [(H-z)^{-1}f](x) = \int_\mathbb{R} \mathcal{R}_z(x,y) \, f(y) \, dy
  \,,
\eeq
where
\beq\label{exprKernel}
\mathcal{R}_z(x,y) := \left\{
\begin{aligned}
  &\frac{1}{k_+(z)+k_-(z)} \, e^{-k_\pm(z)|x|-k_\mp(z)|y|}
  \,,
  &\pm x\geq0\,, \, \pm y\leq0
  \,, 
  \\
  &\frac{1}{2k_\pm(z)} \, e^{-k_\pm(z)|x-y|} 
  &
  \\
  & \quad \pm \frac{k_+(z)-k_-(z)}{2k_\pm(z)\big(k_+(z)+k_-(z)\big)}
  \, e^{-k_\pm(z)|x+y|}
  \,, 
  &\pm x\geq0\,, \, \pm y\geq0
  \,.
\end{aligned}
\right.
\eeq
\end{proposition}
\begin{Remark} 
The kernel $\mathcal{R}_z(x,y)$ is clearly bounded 
for every $(x,y) \in \Real^2$ and fixed $z \not= \pm i$.
Moreover, using~(\ref{expansion}) below, 
it can be shown that it remains bounded for $z = \pm i$ as well.
Hence, contrary to the case of the resolvent kernel 
of the free Hamiltonian~$-\Delta$ in one or two dimensions,
the resolvent kernel of~$H$ has \emph{no local singularity}.
On the other hand, and again contrary to the case of the Laplacian,
for all fixed $(x,y)\in\mathbb{R}^2$,
$|\mathcal{R}_z(x,y)|\longrightarrow+\infty$ 
as $\Re z\rightarrow+\infty$, $z\in\CS$.
Hence, the kernel exhibits a \emph{blow-up at infinity}. 
The absence of singularity will play a fundamental role in the analysis
of weakly coupled eigenvalues in Section~\ref{sBS}. 
Moreover, we shall see in Section~\ref{sPseudo} that the singular behaviour 
at infinity is responsible for the spectral instability of~$H$.
\end{Remark}
\begin{proof}[Proof of Proposition~\ref{propIntegral}]
Let $z\notin [0,\infty) +i\{-1,1\}$ and $f\in L^2(\mathbb{R})$. 
We look for the solution of the resolvent equation $(H-z)u=f$.

The general solutions~$u_\pm$ of the individual equations
\beq\label{eqpm}
  -u''+(\pm i-z)u -f = 0
  \qquad \mbox{in} \qquad
  \mathbb{R}_\pm
  \,,
\eeq
where $\mathbb{R}_+ := [0,+\infty)$ and $\mathbb{R}_-:=(-\infty,0]$,
are given by
\[
  u_\pm(x) = \alpha_\pm(x) \, e^{k_\pm(z)x} + \beta_\pm(x) \, e^{-k_\pm(z)x} 
  \,,
\]
where $\alpha_\pm$, $\beta_\pm$ are functions to be yet determined. 
Variation of parameters leads to the following system:
\[
 \left\{
 \begin{array}{lll} \alpha_\pm'(x)e^{k_\pm(z)x}  
  +  \beta_\pm'(x)e^{-k_\pm(z)x} & = & 0\,, 
  \\
  k_\pm(z)\alpha_\pm'(x)e^{k_\pm(z)x}  
  - k_\pm(z)\beta_\pm'(x)e^{-k_\pm(z)x} & = & -f\,.
 \end{array}
\right.
\]
Hence, we can choose
\begin{align*}
 \alpha_\pm(x) &= -\frac{1}{2k_\pm(z)}\int_0^xf(y)\,e^{-k_\pm(z)y}dy + A_\pm
  \,,& \pm x>0\,,&\\
 \beta_\pm(x) &= \frac{1}{2k_\pm(z)}\int_0^xf(y)\,e^{k_\pm(z)y}dy+B_\pm
  \,, &\pm x>0\,,&
\end{align*}
where $A_\pm,B_\pm$ are arbitray complex constants.
The desired general solutions of~(\ref{eqpm}) are then given by
\beq\label{exprSol}
  u_\pm(x) = \frac{-1}{k_\pm(z)}\int_0^x f(y) \,
  \sinh\big(k_\pm(z)(x-y)\big) \, dy 
  + A_\pm e^{k_\pm(z)x}+B_\pm e^{-k_\pm(z)x}
  \,,
\eeq
with $(A_+,A_-,B_+,B_-)\in\mathbb{C}^4\,$.

Among these solutions, 
we are interested in those which satisfy the regularity conditions
\beq
  u_+(0) = u_-(0)
  \,, \qquad
  u_+'(0) = u_-'(0)
  \,.
\eeq
These conditions are equivalent to the system
\[
 \left\{
 \begin{array}{ccc}
  A_+ + B_+ & = & A_-+ B_-\,,\\
  k_+(z)A_+-k_+(z)B_+ & = & k_-(z)A_- - k_-(z)B_-\,,
 \end{array}
\right.
\]
whence we obtain the following relations:
\beq\label{systemCond}
\left\{
\begin{array}{ccc}
 2A_+ & = & \big(k_+(z)+k_-(z)\big)A_-+ \big(k_+(z)-k_-(z)\big)B_-\,,\\
 2B_+ & = & \big(k_+(z)-k_-(z)\big)A_- + \big(k_+(z)+k_-(z)\big)B_-\,.
\end{array}
\right.
\eeq

Summing up, assuming~\eqref{systemCond}, the function
\beq\label{u.cases}
  u(x) :=
  \begin{cases}
    u_+(x) & \mbox{if} \quad x \geq 0 \,,
    \\
    u_-(x) & \mbox{if} \quad x \leq 0 \,,
  \end{cases}  
\eeq
belongs to $W_\mathrm{loc}^{2,2}(\Real)$
and solves the differential equation~\eqref{eqpm} 
in the whole~$\Real$.
It remains to check some decay conditions as $x\rightarrow\pm\infty$ 
in addition to~(\ref{systemCond}).
This can be done by setting
\begin{align}
  A_+ &:= \frac{1}{2k_+(z)}\int_0^{+\infty}f(y)\,e^{-k_+(z)y}dy
  \,, \label{A+}
  \\
  B_- &:= \frac{1}{2k_-(z)}\int_{-\infty}^0f(y)\,e^{k_-(z)y}dy
  \,. \label{B-}
\end{align}
Indeed, then
\begin{eqnarray*}
 u_+(x) & = & -\frac{1}{2k_+(z)}\,e^{k_+(z)x}\int_x^{+\infty}f(y)\,e^{-k_+(z)y}dy
 \\
 &&+ e^{-k_+(z)x}\left(\frac{1}{2k_+(z)}\int_0^xf(y)\,e^{k_+(z)y}dy + B_+\right)
\end{eqnarray*}
goes to $0$ as $x\rightarrow+\infty$, and similarly for~$u_-$. 

By gathering relations (\ref{systemCond}), (\ref{A+}) and (\ref{B-}), 
we obtain the following values for~$A_-$ and~$B_+$:
\begin{align}
  A_- =\ & \frac{1}{k_+(z)+k_-(z)}
  \int_0^{+\infty}f(y)\,e^{-k_+(z)y}dy
  \nonumber
  \\
  & -\frac{k_+(z)-k_-(z)}{2k_- (z)\big(k_+(z)+k_-(z)\big)}
  \int_{-\infty}^0f(y)\,e^{k_-(z)y}dy
  \,,
  \label{A-}
  \\
  B_+ =\ & \frac{k_+(z)-k_-(z)}{2k_+(z)\big(k_+(z)+k_-(z)\big)}
  \int_0^{+\infty}f(y)\,e^{-k_+(z)y}dy 
  \nonumber
  \\
  &+ \frac{1}{k_+(z)+k_-(z)}
  \int_{-\infty}^0f(y)\,e^{k_-(z)y}dy
  \,.
  \label{B+}
\end{align}
Replacing the constants $A_+,A_-,B_+,B_-$ 
by their values (\ref{A+}), (\ref{A-}), (\ref{B+}) and (\ref{B-}), respectively,
expression~\eqref{u.cases} with~(\ref{exprSol}) gives 
the desired integral representation 
\begin{equation}\label{preSchur}
 u(x) = \int_\mathbb{R} \mathcal{R}_z(x,y) \, f(y) \, dy
\end{equation}
for a decaying solution of the differential equation~\eqref{eqpm} in~$\Real$.

To complete the proof, it remains to check that~$u$ given by~\eqref{preSchur}
is indeed in the operator domain $\Dom(H) = W^{2,2}(\mathbb{R})$. 
Using for instance the Schur test (\cf~\eqref{Schur1} below), 
it is straightforward to check that~$u$ is in $L^2(\mathbb{R})$ 
provided that $f\in L^2(\mathbb{R})$. 
Therefore $u'' = (i\,\textrm{sign~}x-z)u-f \in L^2(\mathbb{R})\,$, 
whence $u\in W^{2,2}(\mathbb{R})$ and $u = (H-z)^{-1}f$.
\end{proof}

This representation of the resolvent will be used 
in Sections~\ref{sPert} and~\ref{sBS} to study the location
of weakly coupled eigenvalues. 
It will also enable us to prove the existence of
non-trivial pseudospectra in Section~\ref{sPseudo}.
In this section we use it to prove 
Proposition~\ref{propSpectrum}.

\begin{proof}[Proof of Proposition~\ref{propSpectrum}]
According to Proposition \ref{propIntegral}, we have
\[
  \sigma(H)\subset \mathbb{R}_++i\,\{-1,+1\} 
  \,.
\]
It remains to prove the inverse inclusion. 
This can be achieved by a standard singular sequence construction.

Let $(a_j)_{j\geq1}$ be a real increasing sequence such that, 
for all $j\geq1$, $a_{j+1}-a_j>2j+1$. 
Let $\xi_j\in C_0^\infty(\mathbb{R})$ be such that 
$\Supp \xi_j\subset(a_j-j,a_j+j)$, 
$\xi_j(x) = 1$ for all $x\in[a_j-1,a_j+1]$, 
and 
\[
  \sup|\xi_j'|\leq \frac{C}{j}
  \,, \qquad 
  \sup|\xi_j''|\leq \frac{C}{j^2}
  \,, 
\]
for some $C>0\,$.

Then, for all $r\geq0$, the sequence
\[
  u_j^\pm(x) := C_j \, \xi_j(\pm x) \, e^{irx}
  \,,
\]
where $C_j$ is chosen so that $\|u_j^\pm\| = 1$, 
is a singular sequence for $H$ corresponding to $z = \pm i + r$ 
in the sense of \cite[Def.~IX.1.2]{Edmunds-Evans}. 
Hence, according to \cite[Thm.~IX.1.3]{Edmunds-Evans}, 
we have
\[
  \sigma(H) \supset
  \mathbb{R}_++i\,\{-1,+1\}
  \,.
\]
This completes the proof of the proposition.
\end{proof}
%

\section{Pseudospectral estimates}\label{sPseudo}
%
The main purpose of this section is to give a proof of Theorem~\ref{thmEstResIntro}.

\begin{proof}[Proof of Theorem~\ref{thmEstResIntro}]
Let $z = \tau+i\delta\,$, where $\tau>0$ and $\delta\in(-1,1)$. 
Recall our convention for the square root 
we fixed at the beginning of Section~\ref{sResSpec}.
The following expansions hold
\begin{equation}\label{expansion}
\begin{aligned}
  k_+(z)
  &= \sqrt{i(1-\delta)-\tau} = i\sqrt{\tau-i(1-\delta)} 
  =  i\sqrt{\tau}+\frac{1-\delta}{2\sqrt{\tau}}
  +\mathcal{O}\left(\frac{1}{|\tau|^{3/2}}\right)
  ,
  \\
  k_-(z)
  &= \sqrt{i(-1-\delta)-\tau} = -i\sqrt{\tau+i(1+\delta)} 
  = -i\sqrt{\tau}+\frac{1+\delta}{2\sqrt{\tau}}
  +\mathcal{O}\left(\frac{1}{|\tau|^{3/2}}\right)
  ,
\end{aligned}
\end{equation}
as $\tau\rightarrow+\infty$.
As a consequence, we have the asymptotics
\begin{eqnarray}
|k_+(z)| \sim \sqrt{\tau}\,, 
&& |k_-(z)| \sim \sqrt{\tau}\,, 
\label{equivk}
\\
\Re k_+(z) \sim \frac{1-\delta}{2\sqrt{\tau}}\,, 
&&\Re k_-(z) \sim \frac{1+\delta}{2\sqrt{\tau}}\,, 
\label{equivRek}\\
|k_+(z) + k_-(z)| \sim \frac{1}{\sqrt{\tau}}\,,
&& |k_+(z)-k_-(z)| \sim 2\sqrt{\tau}\,, 
\label{equivk+k}
\end{eqnarray}
as $\tau\rightarrow+\infty$.

Let us prove the upper bound in~(\ref{EstResIntro}) using the Schur test:
\beq\label{Schur1}
 \|(H-z)^{-1}\|^2 \ \leq \
  \sup_{x\in\mathbb{R}}\int_{\Real}|\mathcal{R}_z(x,y)| \, dy
  \ \cdot \
  \sup_{y\in\mathbb{R}}\int_{\Real}|\mathcal{R}_z(x,y)| \,dx 
  \,.
\eeq
After noticing the symmetry relation
$\mathcal{R}_z(x,y) = \mathcal{R}_z(y,x)$ valid for all $(x,y)\in\mathbb{R}^2$
(which is a consequence of the $\mathcal{T}$-self-adjointness of~$H$),
we simply have
\beq\label{Schur}
  \|(H-z)^{-1}\| \ \leq \ \sup_{x\in\mathbb{R}}\int_{\Real}|\mathcal{R}_z(x,y)| \, dy
  \,.
\eeq
By virtue of~(\ref{exprKernel}), for all $x>0$, 
\begin{align}
  \int_{\Real}|\mathcal{R}_z(x,y)| \, dy  
  \ \leq \ & \frac{1}{|k_+(z)+k_-(z)|}
  \int_{-\infty}^0e^{-\Re k_+(z) \, x+\Re k_-(z) \,y} \, dy 
  \nonumber \\
  & + 
  \frac{1}{2|k_+(z)|}\int_0^{+\infty}e^{-\Re k_+|x-y|} \, dy 
  \nonumber \\
  & + \frac{|k_+(z)-k_-(z)|}{2|k_+(z)||k_+(z)+k_-(z)|}
  \int_0^{+\infty}e^{-\Re k_+(z)(x+y)} \,dy   
  \nonumber \\  
  \ \leq \ & 
  \frac{1}{\Re k_-(z)|k_+(z)+k_-(z)|} 
  + \frac{1}{2\Re k_+(z)|k_+(z)|} 
  \nonumber \\
  & + \frac{|k_+(z)-k_-(z)|}{2\Re k_+(z)|k_+(z)||k_+(z)+k_-(z)|}
  \,. 
  \label{schur+}
\end{align}
Similarly, if $x<0$,
\begin{align}
  \int_{\Real}|\mathcal{R}_z(x,y)| \, dy 
  \ \leq \ & \frac{1}{\Re k_+(z)|k_+(z)+k_-(z)|} + \frac{1}{2\Re k_-(z)|k_-(z)|} 
  \nonumber \\
  & + \frac{|k_+(z)-k_-(z)|}{2\Re k_-(z)|k_-(z)||k_+(z)+k_-(z)|}
  \,. 
  \label{schur-}
\end{align}
According to (\ref{equivk})--(\ref{equivk+k}), 
the right hand sides in~(\ref{schur+}) and~(\ref{schur-}) 
are both equivalent to
\[
  2\tau\big[(1+\delta)^{-1}+(1-\delta)^{-1}\big] \leq \frac{4\tau}{1-|\delta|}
  \,,
\]
whence~(\ref{Schur}) yields the upper bound in~(\ref{EstResIntro}).

In order to get the lower bound, 
we set 
\begin{equation}\label{pseudomode}
  f_0(x) := e^{-\barr{k_+(z)}\,x}\chi_{(0,\infty)}(x)
  \,,
\end{equation}
where~$\chi_\Sigma$ denotes the characteristic function of a set~$\Sigma$. 
Then according to (\ref{exprKernel}),
\begin{align}
  \|(H-z)^{-1}f_0\|^2 
  & \geq \int_{-\infty}^0\left|\frac{1}{k_+(z)+k_-(z)}
  \int_0^{+\infty}e^{k_-(z)\,x-2\Re k_+(z)\,y} \,dy\right|^2dx
  \\
  & = \frac{1}{|k_+(z)+k_-(z)|^2}
  \int_{-\infty}^0 \! e^{2\Re k_-(z) \, x} \, dx 
  \left(\int_0^{+\infty} \!\! e^{-2\Re k_+(z)\,y}\,dy\right)^{\!2}
  \\
  &= \frac{1}{\big(2\Re k_+(z)\big)^2 \,  2\Re k_-(z) \, |k_+(z)+k_-(z)|^2}
  \,.\label{lowerRes}
\end{align}
On the other hand, we have
\begin{equation}\label{f0}
  \|f_0\|^2 = \frac{1}{2\Re k_+(z)}
  \,.
\end{equation}
Hence, using~(\ref{equivRek}) and~(\ref{equivk+k}),
\[
  \frac{\|(H-z)^{-1}f_0\|}{\|f_0\|} 
  \geq \frac{1}{2\sqrt{\Re k_+(z)\Re k_-(z)}\, |k_+(z)+k_-(z)|} 
  \sim \frac{\tau}{\sqrt{1-\delta^2}}
\]
as $\tau\rightarrow+\infty\,$, and the lower bound in (\ref{EstResIntro}) follows.
\end{proof}
\begin{Remark}[Irrelevance of discontinuity]\label{RemSingularity}
Although the proof above relies on the particular form of 
the potential $i\,\sgn(x)$, 
it turns out that the discontinuity at $x=0$ is not responsible 
for the spectral instability highlighted by Theorem~\ref{thmEstResIntro}. 
Indeed, consider instead of the potential $i\,\sgn(x)$
a smooth potential $V(x)$ such that, for some $a>0$, the difference
 \[
  h(x):=i\,\sgn(x)-V(x)
 \]
 is supported in the interval $[-a,0]$. 
In order to get a lower bound for 
the norm of the resolvent of the regularised operator 
$\tilde H := -\frac{d^2}{dx^2}+V(x)$,
 we shall use the pseudomode
 \[
  g_0 := (H-z)^{-1}f_0\,,
 \]
where the function~$f_0$ is introduced in~\eqref{pseudomode}. 
Using again the asymptotic expansions (\ref{expansion}),
one can check that, provided that $\Re z$ is large enough,
\[
 \|hg_0\|^2 \leq C\,(\Re z)^2
\]
for some $C>0$ independent of~$z$. 
Thus, in view of (\ref{f0}), we have
\[
 \|(\tilde H-z)g_0\|\leq\|f_0\|+\|hg_0\| = \mathcal{O}(\Re z)
\]
as $\Re z\rightarrow+\infty$, $z\in\CS$.
On the other hand, (\ref{lowerRes}) yields 
\[
 \|g_0\|^2\geq C'\, (\Re z)^{5/2}
\]
for some $C'>0$ independent of~$z$. 
Consequently, $g_0$ is a $(\Re z)^{-1/4}$-pseudomode for $\tilde H-z$, 
or more specifically,
\begin{equation}\label{estSmooth}
\|(\tilde H-z)^{-1}\|\geq c \, (\Re z)^{1/4}
\end{equation}
with $c>0$ independent of~$z$, 
as $\Re z\rightarrow+\infty$, $z\in\CS$.

Summing up, 
despite of the fact that the lower bound in~(\ref{estSmooth}) 
is not as good as that of Theorem~\ref{thmEstResIntro},
the presence of non-trivial pseudospectra 
for the operator~$\tilde H$ clearly indicates 
that the discontinuity of the potential $i\,\sgn(x)$ 
does not really play any essential role
in the spectral instability of~$H$.
\end{Remark}
%

\section{General properties of the perturbed operator}\label{sPert}
%
In this section, we state some basic properties about 
the perturbed operator~$H_\eps$ introduced in~\eqref{operator.weak}.
Here~$\eps$ is not necessarily small and positive. 

\subsection{Definition of the perturbed operator}\label{ssDefA+V}
The unperturbed operator~$H$ 
introduced in~\eqref{operator} is associated 
(in the sense of the representation theorem \cite[Thm.~VI.2.1]{Kato})
with the sesquilinear form
$$
\begin{aligned}
  h(\psi,\phi) &:= \int_\mathbb{R}\psi'(x)\bar \phi'(x)\,dx 
  + i \int_0^{+\infty}\psi(x)\bar \phi(x)\,dx
  - i \int_{-\infty}^0\psi(x)\bar \phi(x)\,dx
  \,, 
  \\
  \Dom(h) &:= W^{1,2}(\Real)
  \,.
\end{aligned}
$$
In view of~\eqref{Num}, $h$~is sectorial with vertex~$-1$ and semi-angle~$\pi/4$.
In fact, $h$~is obtained as a bounded perturbation of the non-negative
form~$q$ associated with the free Hamiltonian~$-\Delta$, 
$$
\begin{aligned}
  q(\psi,\phi) &:= \int_\mathbb{R}\psi'(x)\bar \phi'(x)\,dx 
  \,, 
  \\
  \Dom(q) &:= W^{1,2}(\Real)
  \,.
\end{aligned}
$$

Given any function $V\in L^1(\mathbb{R})$, 
let~$v$ be the sesquilinear form of the corresponding multiplication operator
(that we also denote by~$V$), \ie,
\[
\begin{aligned}
  v(\psi,\phi) &:= \int_\mathbb{R}V(x)\psi(x)\bar \phi(x)\,dx
  \,,
  \\
  \Dom(v) &:= \left\{\psi\in L^2(\mathbb{R}) : 
  \ |V|^{1/2}\psi\in L^2(\mathbb{R})\right\}
  \,.
\end{aligned}
\]
As usual, we denote by $v[\psi]:=v(\psi,\psi)$ the corresponding quadratic form. 

\begin{lemma}\label{Lem.relative}
Let $V\in L^1(\mathbb{R})$. Then $\Dom(v) \supset W^{1,2}(\Real)$
and, for every $\psi \in W^{1,2}(\Real)$,
\begin{equation}
  |v[\psi]| \leq 2 \|V\|_{L^1(\Real)} \|\psi'\| \|\psi\|  
  \,.
\end{equation}
\end{lemma}
\begin{proof}
Set $f(x) := \int_{-\infty}^x V(\xi) d\xi$.
For every $\psi \in C_0^\infty(\Real)$, 
an integration by parts together with the Schwarz inequality yields
\begin{align*}
  |v[\psi]| 
  &= \left| \int_\Real f'(x) |\psi(x)|^2 \, dx \right|
  = \left| \int_\Real f(x) \, 2 \Re\big(\psi'(x) \bar\psi(x)\big) \, dx \right|
  \\
  &\leq 2 \|V\|_{L^1(\Real)} \|\psi'\| \|\psi\|  
  \,.
\end{align*}
By density of $C_0^\infty(\Real)$ in $W^{1,2}(\Real)$,
the inequality extends to all $\psi \in W^{1,2}(\Real)$
and, in particular, $|v[\psi]| < \infty$ whenever $\psi \in W^{1,2}(\Real)$.
\end{proof}

It follows from the lemma that~$v$ is $\frac{1}{2}$-subordinated to~$q$,
which in particular implies that~$v$ is relatively bounded with respect to~$q$  
with the relative bound equal to zero.
Classical stability results (see, \eg, \cite[Sec.~5.3.4]{KS-book})
then ensure that the form $q+v$ is sectorial and closed.
Since~$h$ is a bounded perturbation of~$q$,
we also know that $h_1 := h+v$ is sectorial and closed.
We define~$H_1$ to be the m-sectorial operator 
associated with the form~$h_1$.
The representation theorem yields 
\begin{equation}\label{H1.weak}
\begin{aligned}
  H_1\psi &= -\psi''+i \sgn \psi + V\psi \,,
  \\
  \Dom(H_1) 
  &= \left\{
  \psi \in W^{1,2}(\Real) : \ \exists\eta \in \sii(\Real), 
  \ \forall \phi\in W^{1,2}(\Real), \
  h_1(\psi,\phi) = (\eta,\phi)
  \right\}
  \\
  &= \left\{
  \psi \in W^{1,2}(\Real) : \ -\psi'' + V\psi \in \sii(\Real)
  \right\}
  \,,
\end{aligned}
\end{equation}
where $-\psi'' + V\psi$ should be understood as a distribution.
By the replacement $V \mapsto \eps V$, 
we introduce in the same way as above the form $h_\eps := h+\eps v$
and the associated operator~$H_\eps$ for any $\eps \in \Real$.
Of course, we have $H_0=H$.

\subsection{The Birman-Schwinger principle}\label{ssBS}
As regards spectral theory,
$H_\eps$~represents a singular perturbation of~$H$,
for we are perturbing an operator with purely essential spectrum.   
An efficient way to deal with such problems in self-adjoint settings
is the method of the \emph{Birman-Schwinger principle},
due to which a study of discrete eigenvalues of
the differential operator~$H_\eps$
is transferred to a spectral analysis of an integral operator.
We refer to~\cite{Birman_1961,Schwinger_1967} for the original works
and to~\cite{SiQF,Si,BGS,Klaus} for an extensive development of the method
for Schr\"odinger operators.
In recent years, the technique has been also applied to Schr\"odinger
operators with complex potentials 
(see, \eg, \cite{Abramov-Aslanyan-Davies_2001,Laptev-Safronov_2009,Frank_2011}).
However, our setting differs from all the previous works 
in that the unperturbed operator~$H$ 
is already non-self-adjoint and its resolvent kernel substantially
differs from the resolvent of the free Hamiltonian. 
The objective of this subsection is to carefully establish 
the Birman-Schwinger principle in our unconventional situation.

In the following, given $V\in L^1(\mathbb{R})$, we denote 
\[
  V_{1/2}(x) := |V|^{1/2}e^{i\arg V(x)}
  \,,
\]
so that $V = |V|^{1/2}V_{1/2}$.

We have introduced~$H$ as an unbounded operator with domain $\Dom(H) = W^{2,2}(\Real)$
acting in the Hilbert space $\sii(\Real)$.
It can be regarded as a bounded operator from $W^{2,2}(\Real)$ to $\sii(\Real)$. 
More interestingly, using the variational formulation, 
$H$~can be also viewed as a bounded operator 
from $W^{1,2}(\mathbb{R})$ to $W^{-1,2}(\mathbb{R})$,
by defining $H\psi$ for all $\psi\in W^{1,2}(\mathbb{R})$ by
\[
 \forall \phi\in W^{1,2}(\mathbb{R})
  \,,\qquad 
  {}_{-1}\sca{H\psi}{\phi}_{+1} := h(\psi,\phi)
  \,,
\]
where ${}_{-1}\sca{\cdot}{\cdot}_{+1}$ denotes the duality bracket between 
$W^{-1,2}(\mathbb{R})$ and $W^{1,2}(\mathbb{R})$.

Similarly, in addition to regarding the multiplication operators 
$|V|^{1/2}$ and $V_{1/2}$ as operators 
from $W^{1,2}(\mathbb{R})$ to $L^2(\mathbb{R})$,
we can view them as operators from $L^2(\mathbb{R})$ to $W^{-1,2}(\mathbb{R})$, 
due to the relative boundedness of~$v$ with respect to~$q$ 
(\cf~Lemma~\ref{Lem.relative} and the text below it).

Finally, let us notice that, for all $z\in\mathbb{C}\setminus\sigma(H)$, 
the resolvent $(H-z)^{-1}$ can be viewed as an operator 
from $W^{-1,2}(\mathbb{R})$ to $W^{1,2}(\mathbb{R})$.
Indeed, for all $\eta\in W^{-1,2}(\mathbb{R})$, 
there exists a unique $\psi\in W^{1,2}(\mathbb{R})$ such that
\begin{equation}\label{resolvent.weak}
  \forall \phi\in W^{1,2}(\mathbb{R})
  \,, \qquad 
  {}_{-1}\sca{\eta}{\phi}_{+1} = h(\psi,\phi) - z (\psi,\phi)
  \,,
\end{equation}
where $(\cdot,\cdot)$ denotes the inner product in~$\sii(\Real)$.
Hence the operator 
$(H-z) : W^{1,2}(\mathbb{R}) \to W^{-1,2}(\mathbb{R})$ 
is bijective.

With the above identifications, 
for all $z \in \Com\setminus\sigma(H)$,
we introduce 
\begin{equation}\label{BS-operator}
  K_z := |V|^{1/2}(H-z)^{-1}V_{1/2}
\end{equation}
as a bounded operator on $\sii(\Real)$ to $\sii(\Real)$.
$K_z$~is an integral operator with kernel
\begin{equation}\label{BS-kernel}
  \mathcal{K}_z(x,y) := |V|^{1/2}(x) \, \mathcal{R}_z(x,y) \, V_{1/2}(y)
  \,,
\end{equation}
where~$\mathcal{R}_z$ is the kernel of the resolvent $(H-z)^{-1}$ 
written down explicitly in~(\ref{exprKernel}).
The following result shows that~$K_z$ is in fact compact. 
\begin{lemma}\label{Lem.BS}
Let $V\in L^1(\mathbb{R})$.
For all $z \in \Com\setminus\sigma(H)$,
$K_z$ is a Hilbert-Schmidt operator.
\end{lemma}
\begin{proof}
By definition of the Hilbert-Schmidt norm, 
\begin{equation}\label{estBSinter}
\begin{aligned}
  \|K_z\|_\mathrm{HS} 
  &= \int_{\mathbb{R}^2}|V(x)||\mathcal{R}_z(x,y)|^2|V(y)| \, dx \, dy 
  \\
  &\leq \|V\|_{L^1(\Real)}^2
  \sup_{(x,y)\in\mathbb{R}^2}|\mathcal{R}_z(x,y)|^2 
  \,.
\end{aligned}
\end{equation}
According to (\ref{exprKernel}), we have
\begin{eqnarray*}
 \lefteqn{\sup_{(x,y)\in\mathbb{R}^2}|\mathcal{R}_z(x,y)|^2} 
  \\
  && 
  \leq \frac{1}{|k_+(z)+k_-(z)|^2} + \left(\frac{1}{|k_+(z)|^2} 
  + \frac{1}{|k_-(z)|^2}\right)
  \left(1+\frac{|k_+(z)-k_-(z)|^2}{|k_+(z)+k_-(z)|^2}\right)
  \,,
\end{eqnarray*}
where the right hand side is finite for all $z\in\mathbb{C}\setminus\sigma(H)$.
\end{proof}

We are now in a position to state the Birman-Schwinger principle 
for our operator~$H_\eps$.
\begin{theorem}[Birman-Schwinger principle]\label{thmBirmanSchwinger}
Let $V\in L^1(\mathbb{R})$ and $\eps \in \Real$.
For all $z\in\mathbb{C}\setminus\sigma(H)$, we have
\[
  z\in\sigma_\mathrm{p}(H_\eps) 
  \quad\Longleftrightarrow\quad 
  -1\in\sigma(\eps K_z)
  \,.
\]
\end{theorem}
\begin{proof}
Clearly, it is enough to establish the equivalence for $\eps=1$.

If $z \in \sigma_\mathrm{p}(H_1)$, then there exists a non-trivial
function $\psi \in \Dom(H_1)$ such that $H_1\psi=z\psi$.
In particular, $\psi \in \Dom(h_1)=W^{1,2}(\Real)$ and 
\begin{equation}\label{ev-weak}
  h_1(\psi,\phi) \equiv h(\psi,\phi) + v(\psi,\phi) = z (\psi,\phi) 
\end{equation}
holds for every $\phi \in W^{1,2}(\Real)$.
We set $g:=|V|^{1/2}\psi \in \sii(\Real)$.
Given an arbitrary test function $\varphi \in \sii(\Real)$,
we introduce an auxiliary function 
$\eta := (H^*-\bar{z})^{-1}|V|^{1/2}\varphi \in W^{1,2}(\Real)$.
(Note that $\sigma(H^*)=\sigma(H)$ and that the spectrum 
is symmetric with respect to the real axis,
so the resolvent $(H^*-\bar{z})^{-1}$ is well defined.
Moreover, recall that $H$~is $\mathcal{T}$-self-adjoint.)
We have
$$
\begin{aligned}
  (K_z g,\varphi) 
  &= v(\psi,\eta) 
  \\
  &= - h(\psi,\eta) + z (\psi,\eta)
  = \overline{- h^*(\eta,\psi) + \bar{z} (\eta,\psi)}
  \\
  &= - \overline{{}_{-1}\sca{|V|^{1/2}\varphi}{\psi}_{+1}}
  \\
  &= - \overline{(\varphi,|V|^{1/2}\psi)}
  \\
  &= - (g,\varphi)
  \,.
\end{aligned}
$$
Here the first equality uses the integral representation~\eqref{BS-kernel} of~$K_z$,
the second equality is due to~\eqref{ev-weak}
and the equality on the third line is a version of~\eqref{resolvent.weak} for~$H^*$.
Hence, $g$~is an eigenfunction of~$K_z$ corresponding to the eigenvalue~$-1$.

Conversely, if $-1\in\sigma(K_z)$, 
then~$-1$ is an eigenvalue of~$K_z$, 
because~$K_z$ is compact (\cf~Lemma~\ref{Lem.BS}).
Hence, there exists a non-trivial $g \in \sii(\Real)$ 
such that $K_z g = -g$.
Defining, $\psi := (H-z)^{-1} V_{1/2} \, g \in W^{1,2}(\Real)$,
we have 
$$
\begin{aligned}
  h_1(\psi,\phi)
  &= h(\psi,\phi) - z (\psi,\phi) + z (\psi,\phi) + v(\psi,\phi) 
  \\
  &= {}_{-1}\sca{V_{1/2} \, g}{\psi}_{+1} + z (\psi,\phi) 
  + {}_{-1}\sca{V\psi}{\phi}_{+1}
  \\
  &= {}_{-1}\sca{V_{1/2} \, g}{\psi}_{+1} + z (\psi,\phi) 
  + {}_{-1}\sca{V_{1/2}K_z g}{\phi}_{+1}
  \\ 
  &= z (\psi,\phi) 
\end{aligned}
$$
for all $\phi \in W^{1,2}(\Real)$,
where the eigenvalue equation is used in the last equality.
It follows that $\psi \in \Dom(H)$ (\cf~\eqref{H1.weak})
and $H\psi=z\psi$. 
\end{proof}

\subsection{Stability of the essential spectrum}
As the last result of this section, we locate the essential spectrum 
of the perturbed operator~$H_\eps$. 

Since there exist various definitions of the essential spectrum
for non-self-adjoint operators 
(\cf~\cite[Sec.~IX]{Edmunds-Evans} or \cite[Sec.~5.4]{KS-book}),
we note that we use the widest (that due to Browder) in this paper.
More specifically, given a closed operator~$T$ in a Hilbert space~$\mathcal{H}$, 
we set 
$
  \sigma_{\mathrm{ess}}(T) 
  := \sigma(T) \setminus \sigma_{\mathrm{disc}}(T)
$, 
where the discrete spectrum is defined as the set of isolated 
eigenvalues~$\lambda$ of~$T$ which have finite algebraic multiplicity
and such that $\Ran(T-\lambda)$ is closed in~$\mathcal{H}$.

Our stability result will follow from the following compactness property.
\begin{lemma}\label{Lem.compact}
Let $V\in L^1(\mathbb{R})$ and $\eps \in \Real$.
For all $z\in\mathbb{C}\setminus[\sigma(H)\cup\sigma(H_\eps)]$,
the resolvent difference
$
  (H_\eps-z)^{-1}-(H-z)^{-1}
$
is a compact operator in $\sii(\Real)$.
\end{lemma}
\begin{proof}
It is straightforward to verify the resolvent equation
$$
  (H_\eps-z)^{-1}-(H-z)^{-1} = - \eps A^* B
  \,,
$$
where
$$
  A:= \overline{V}_{1/2} (H_\eps^*-\bar{z})^{-1}
  \qquad \mbox{and} \qquad 
  B := |V|^{1/2} (H-z)^{-1}
$$
are bounded operators 
(recall that $\Dom(h_\eps) = W^{1,2}(\Real) \subset \Dom(v)$).
It is thus enough to show that~$B$ is compact. 
It is equivalent to proving that $BB^*$ is compact.
However, $BB^*$~is an integral operator with kernel
$$
  |V|^{1/2}(x) \, \mathcal{N}_z(x,y) \,|V|^{1/2}(y)
  \,,
$$
where 
$$
  \mathcal{N}_z(x,y) :=
  \int_\mathbb{R}\mathcal{R}_z(x,\xi) \, \barr{\mathcal{R}_z(y,\xi)} \, d\xi
$$ 
is the integral kernel of $(H-z)^{-1}(H^*-\bar z)^{-1}$.
Consequently, 
\begin{equation}\label{estBB*}
  \|BB^*\|_\mathrm{HS} \leq \|V\|_{L^1(\Real)} 
  \sup_{(x,y)\in\mathbb{R}^2}|\mathcal{N}_z(x,y)|
  \,.
\end{equation}
Using~(\ref{exprKernel}), 
it is straightforward to check that, 
for all $z\in\Com\setminus\sigma(H)$, 
$\mathcal{R}_z\in L^\infty\big(\mathbb{R} ; L^2(\mathbb{R})\big)$, 
and thus the supremum on the right-hand side of~(\ref{estBB*}) 
is a finite ($z$-dependent) constant.
Summing up, $BB^*$ is Hilbert-Schmidt, in particular it is compact.
\end{proof}
\begin{proposition}\label{propStabEss}
Let $V\in L^1(\mathbb{R})$. For all $\eps \in \Real$, we have
\begin{equation}\label{ess.stability}
  \sigma_\mathrm{ess}(H_\eps) = \sigma_\mathrm{ess}(H) =  \mathbb{R}_++i\,\{-1,+1\}
  \,.
\end{equation}
\end{proposition}
\begin{proof}
First of all, notice that, since~$H_\eps$ is m-sectorial for all $\eps \in \Real$, 
the intersection of the resolvent sets of~$H_\eps$ and~$H$ is not empty.
By Lemma~\ref{Lem.compact} and a classical stability result 
about the invariance of the essential spectra under perturbations
(see, \eg, \cite[Thm.~IX.2.4]{Edmunds-Evans}),
we immediately obtain~\eqref{ess.stability} for 
more restrictive definitions of the essential spectrum.
To deduce the result for our definition of the essential spectrum, 
it is enough to notice that the exterior of $\sigma_\mathrm{ess}(H)$
is connected (\cf~\cite[Prop.~5.4.4]{KS-book}).
\end{proof}
\begin{Remark}
In view of Proposition~\ref{propStabEss}, 
the equivalence of Theorem~\ref{thmBirmanSchwinger} remains to hold
if $\sigma_\mathrm{p}(H_\eps)$ is replaced by 
$\sigma(H_\eps)$ or $\sigma_\mathrm{disc}(H_\eps)$.
\end{Remark}
%

\section{Eigenvalue estimates}\label{sBS}
%
In this section, we consecutively prove 
Theorems~\ref{thmSpecBound1} and~\ref{thmSpecBound2}.

\subsection{Proof of Theorem~\ref{thmSpecBound1}}\label{ssBS1}
Our strategy is based on Theorem~\ref{thmBirmanSchwinger}
and on estimating the norm of the Birman-Schwinger operator~$K_z$
by its Hilbert-Schmidt norm.
To get a better estimate than that of~\eqref{estBSinter},
we proceed as follows.

Let us partition the complex plane into several regions where
$z \mapsto \mathcal{R}_z$ has a different behaviour. 
We set
\[
 \begin{aligned}
  D_+ & :=  
  \big\{z\in\mathbb{C} : \, |z-i|\leq 3/2\big\}\setminus\big(\mathbb{R}_++i\big)
  \,, \\
  D_- & :=  
  \big\{z\in\mathbb{C} : \, |z+i|\leq 3/2\big\}\setminus\big(\mathbb{R}_+-i\big)
  \,, \\
  U & :=   \mathbb{C}\setminus\big(\bar\CS\cup D_+\cup D_-\big)
  \,, \\
  W  & :=   \CS\setminus\big(D_+\cup D_-\big)
  \,,
 \end{aligned}
\]
where $\CS$ is defined in~(\ref{Num}), see Figure~\ref{figSub}.
We have indeed
\[
 \mathbb{C}\setminus\big(\mathbb{R}_++i\{-1,1\}\big) = D_+\cup D_-\cup U\cup W\,.
\]

\begin{figure}[h]\begin{center}
\includegraphics[width=0.8\textwidth]{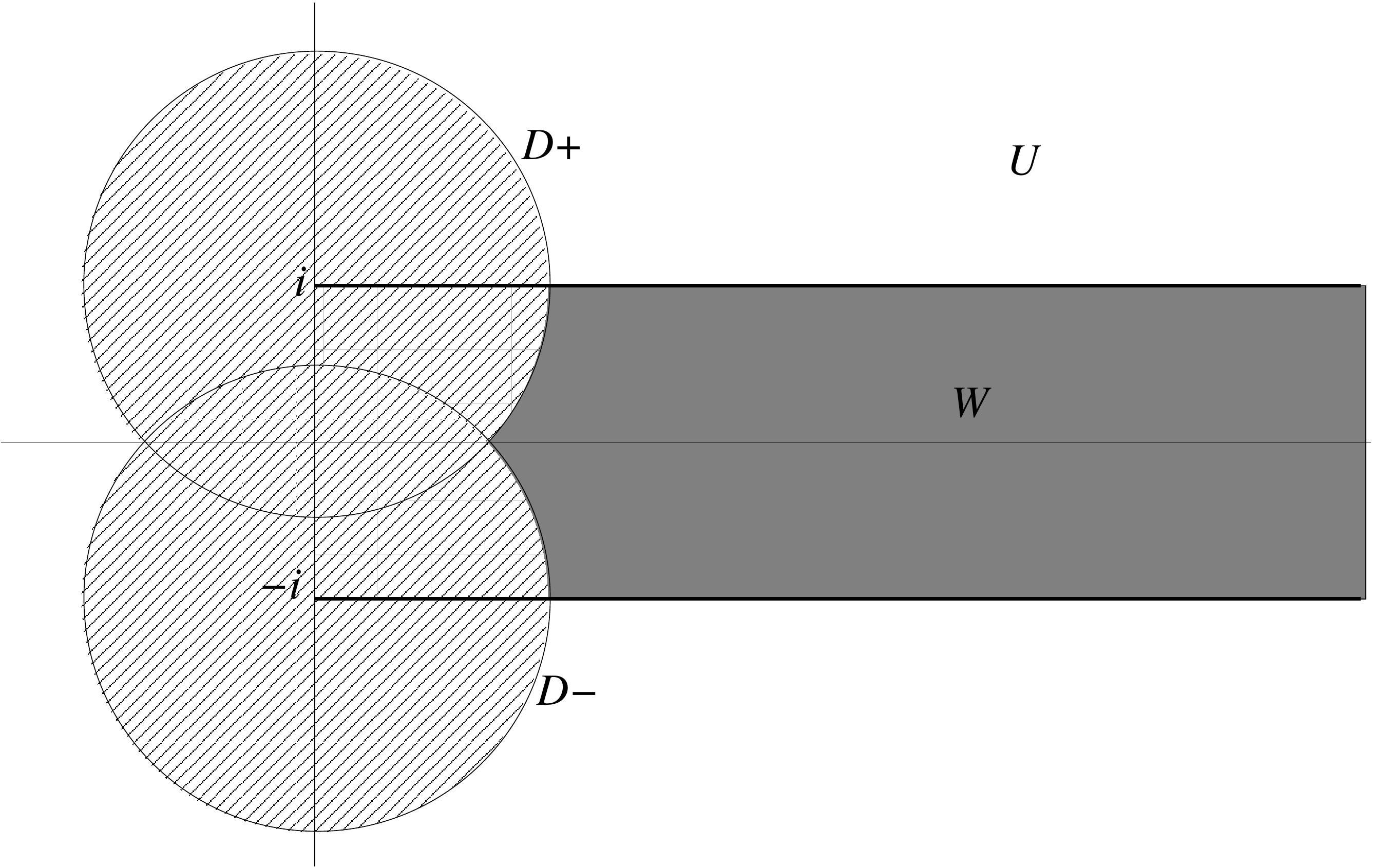}
\caption{The subdomains $D_+$, $D_-$, $U$ and $W$.}\label{figSub}
\end{center}
\end{figure}

First, let us estimate $\sup_{\mathbb{R}^2} |\mathcal{R}_z|$ for $z\in D_+$.
As $z\to i$, we have $k_+(z)\to 0$ and $k_-(z)\to\sqrt{-2i}$.
Thus, there exist positive constants $c_0\,$, $c_1$ and $c_2$ such that, 
for all $z\in D_+$, 
\beq
|k_+(z)+k_-(z)|\geq\frac{1}{c_0}
\,, \qquad
|k_+(z)-k_-(z)|\leq c_1
\,, \qquad
|k_-(z)|\geq\frac{1}{c_2}\,.
\eeq
According to~(\ref{exprKernel}), we then have, 
for all $(x,y)\in\mathbb{R}^2$ such that $xy\leq0$,
\beq\label{estBS+1}
|\mathcal{R}_z(x,y)| \leq \frac{1}{|k_+(z)+k_-(z)|} \leq c_0\,,
\eeq
and, for all $(x,y)\in \big\{x\leq0,y\leq0\big\}$,
\beq\label{estBS+2}
|\mathcal{R}_z(x,y)| \leq \frac{1}{2|k_-(z)|}
\left(1+\frac{|k_+(z)-k_-(z)|}{|k_+(z)+k_-(z)|}\right) 
\leq \frac{c_2}{2}(1+c_0c_1)\,.
\eeq
It remains to check that there is no singularity as $z\rightarrow i$ 
for $x>0\,$, $y>0\,$:
\begin{align}
 |\mathcal{R}_z(x,y)| 
  & = \frac{1}{2|k_+(z)|}\left|e^{-k_+(z)|x-y|} 
  + \Big(-1+\frac{2k_+(z)}{k_+(z)+k_-(z)}\Big)e^{-k_+(z)(|x|+|y|)}\right|
  \nonumber \\
  & \leq  \frac{1}{2|k_+(z)|}\left|e^{-k_+(z)|x-y|} -e^{-k_+(z)(|x|+|y|)}\right| 
  + \frac{1}{|k_+(z)+k_-(z)|}
  \nonumber\\
  & \leq  c_0 + \frac{1}{2|k_+(z)|}\left|\Big(e^{-k_+(z)|x-y|}-1\Big) 
  -\Big(e^{-k_+(z)(|x|+|y|)}-1\Big)\right|
  \nonumber \\
  & \leq c_0 + \frac{|x-y|+|x|+|y|}{2}
  \nonumber\\
  & \leq c_0+|x|+|y|\,,\label{estBS+3}
\end{align}
where we have used the inequality 
$|e^{-\omega}-1|\leq |\omega|$ for $\Re\omega\geq0$.
Using (\ref{estBS+1}), (\ref{estBS+2}) and (\ref{estBS+3}), 
we then get, for all $z\in D_+$,
\begin{align}
  \|K_z\|_\mathrm{HS}^2 
  & \leq \int_{\mathbb{R}^2}|V(x)|\Big(3c_0^2+\frac{c_2^2}{4}(1+c_0c_1)^2 
  + 2\big(|x|+|y|\big)^2\Big)|V(y)| \,dx \, dy 
  \nonumber\\
  & \leq C_+\left(\int_\mathbb{R}(1+|x|^2)|V(x)| \, dx\right)^2\,, 
  \label{estBS+}
\end{align}
with some $C_+>0$.

Similarly, one can check that there exists $C_->0$ such that, for all $z\in D_-$,
\beq\label{estBS-}
  \|K_z\|_\mathrm{HS}^2 
  \leq C_-\left(\int_\mathbb{R}(1+|x|^2)|V(x)| \, dx\right)^2\,.
\eeq

Now let us consider the region~$U$. 
Notice that, as $|z|\rightarrow+\infty$, $z\in U$, we have
\[
  k_+(z)-k_-(z)\longrightarrow 0
  \qquad\mbox{and}\qquad
  k_+(z)\sim k_-(z)\sim \sqrt{-z}
  \,,
\]
hence $|k_++k_-|^{-1}$, $|k_+|^{-1}$, $|k_-|^{-1}$ and $|k_+-k_-|$ 
are uniformly bounded in~$U$. 
Thus, there exists $C_1>0$ such that, for all $z\in U$,
\begin{equation}\label{estBSU}
  \|K_z\|_\mathrm{HS}^2 
  \leq  \|V\|_{L^1(\Real)}^2\sup_{(x,y)\in\mathbb{R}^2}|\mathcal{R}_z(x,y)|^2 
  \leq  C_1\|V\|_{L^1(\Real)}^2\,.
\end{equation}

Finally, for $z\in W$, 
we use the asymptotic expansions~(\ref{equivk}) and~(\ref{equivk+k}). 
In particular, there exist $c_3>0$, $c_4>0$ and $c_5>0$ such that,
for all $z\in W$,
\[
  2|k_\pm(z)|\geq\frac{\sqrt{\Re z}}{c_3}
  \,, \quad
  |k_-(z)-k_+(z)|\leq c_4\sqrt{\Re z}
  \,, \quad
  |k_+(z)+k_-(z)|\geq\frac{1}{c_5\sqrt{\Re z}}
  \,.
\]
Thus, according to~(\ref{exprKernel}), we have
\[
  \sup_{(x,y)\in\mathbb{R}^2}|\mathcal{R}_z(x,y)|
  \leq\frac{c_3}{\sqrt{\Re z}}+c_3c_4c_5\sqrt{\Re z} \leq \sqrt{C_2\Re z}
\]
for some $C_2>0$, hence
\beq\label{estBSV}
  \|K_z\|_\mathrm{HS}^2 \leq C_2 \, \Re z \, \|V\|_{L^1(\Real)}^2
  \,.
\eeq

Gathering (\ref{estBS+}), (\ref{estBS-}), (\ref{estBSU}) and (\ref{estBSV}), 
we obtain, for all $z\in\mathbb{C}\setminus\big(\mathbb{R}_++i\{-1,+1\}\big)$, 
\beq\label{estBS1}
  \|K_z\|_\mathrm{HS}^2 
  \leq \max\Big(\max(C_+,C_-,C_1) \big\|(1+|\cdot|^2)V\big\|_{L^1(\Real)}^2\,, 
  C_2 \, \Re z \, \|V\|_{L^1(\Real)}^2\Big)
  \,,
\eeq
and more precisely when $z\notin \CS$, 
\[
 \|K_z\|_\mathrm{HS}^2 \leq \max(C_+,C_-,C_1)
  \big\|(1+|\cdot|^2)V\big\|_{L^1(\Real)}^2\,\,.
\]
In particular, 
if $\|(1+|\cdot|^2)V\|_{L^1(\Real)}^2<\max(C_+,C_-,C_1)^{-1}$ 
and either $z\notin\CS$ or $\Re z<(C_2\,\|V\|_{L^1(\Real)}^2)^{-1}$, 
then $\|K_z\|_\mathrm{HS}<1$ and~$-1$ cannot be 
in the spectrum of~$K_z$. 
After the replacement $V \mapsto \eps V$,
we therefore get Theorem~\ref{thmSpecBound1} 
as a consequence of Theorem~\ref{thmBirmanSchwinger}.
\hfill\qed

\subsection{Proof of Theorem~\ref{thmSpecBound2}}\label{ssBS2}
Let~$V$ satisfy the assumptions of Theorem~\ref{thmSpecBound2}
with $n\geq2$ and $\varepsilon>0$.
The present proof is again based on Theorem~\ref{thmBirmanSchwinger},
but we use a more sophisticated estimate of the norm of~$K_z$
for which the extra regularity hypotheses are needed.

The first step in our proof is to isolate the singular part 
of the kernel~$\mathcal{K}_z$.
The idea comes back to~\cite{Si}, where the singularity of
the free resolvent $(-\Delta-z)^{-1}$ at $z=0$ is singled out.
In the present setting, however,
the resolvent $(H-z)^{-1}$ is rather singular as $\Re z\rightarrow+\infty$.
In other words, we want to find a decomposition of the form
\beq\label{decompK}
  K_z = L_z+M_z
  \,,
\eeq
where $\|L_z\|\to+\infty$ as $\Re z\to+\infty$, 
while $M_z$ stays uniformly bounded with respect to $z$. 
The integral kernels of~$L_z$ and~$M_z$ 
will be denoted by~$\mathcal{L}_z$ and~$\mathcal{M}_z$, respectively. 

Notice that it is enough to consider $z\in\CS$ since,
according to Theorem~\ref{thmSpecBound1}, 
every eigenvalue of~$H_\eps$ belongs to 
the half-strip~$\CS$ provided that~$\varepsilon$ is small enough.

In this paper, 
motivated by the asymptotic expansions~\eqref{expansion},
we use the decomposition~\eqref{decompK} 
with the singular part~$L_z$ given by the integral kernel
\beq\label{defL}
  \mathcal{L}_z(x,y) := 
  \sqrt{\Re z} \, |V|^{1/2}(x) \, e^{-i\sqrt{\Re z}\,(x+y)} \, V_{1/2}(y)
  \,.
\eeq
Properties of~$M_z$ are then stated in the following lemma.

\begin{lemma}\label{lemDecompK}
For all $z\in\CS$ and $(x,y)\in\mathbb{R}^2$, 
the integral kernel of the operator~$M_z$ defined by~\eqref{decompK}
with~\eqref{defL} satisfies
\beq\label{defM}
  \mathcal{M}_z(x,y) 
  = \frac{1}{2} |V|^{1/2}(x) e^{-i\sqrt{\Re z}\,(x+y)}
  \big[\Im z \, (x+y)-(|x|+|y|)\big]V_{1/2}(y) + m_z(x,y)
  \,,
\eeq
where for some $k>0\,$, the function $m_z$ satisfies, 
for all $z\in\mathcal{S}$ such that $\Re z\geq1\,$,
\beq\label{def.m}
  |m_z(x,y)|\leq \frac{k}{\sqrt{\Re z}} \, |V|^{1/2}(x)\,(1+x^2+y^2)\,|V|^{1/2}(y)
  \,.
\eeq
If $V\in L^1\big(\mathbb{R},(1+x^4)\,dx\big)$, 
then $\|M_z\|_\mathrm{HS}$ is uniformly bounded with respect to $z\in\CS$.
\end{lemma}
\begin{proof}
In the following computations we assume $\Re z\geq1\,$.

First, let $x\geq0$ and $y\leq0$. 
Then, according to~(\ref{exprKernel}) 
and the asymptotic behaviour of $k_+(z)$ and $k_-(z)$ given in~(\ref{expansion}),
\begin{equation*}
  \mathcal{R}_z(x,y) = \frac{1}{k_+(z)+k_-(z)} \, e^{-k_+(z)\,x+k_-(z)\,y}
  = e^{-k_+(z)\,x+k_-(z)\,y}\big(\sqrt{\Re z} + \delta_1(z)\big)
  \,,
\end{equation*}
where $\delta_1(z)$ does not depend on $(x,y)$ and 
$
  \delta_1(z) = \mathcal{O}(1/\sqrt{\Re z})
$.
Thus,
\begin{align}
  \mathcal{M}_z(x,y) 
  = \ & \sqrt{\Re z}\,|V|^{1/2}(x) \, e^{-i\sqrt{\Re z}\,(x+y)}
  \left(e^{\Lambda_z(x,y)}-1\right)V_{1/2}(y) 
  \nonumber \\
  & + \delta_1(z) \, |V|^{1/2}(x) \, e^{-k_+(z)\,x+k_-(z)\,y} \, V_{1/2}(y)
  \,, \label{decompKinter}
\end{align}
where
\[
  \Lambda_z(x,y) := 
  \left(-k_+(z)+i\sqrt{\Re z}\right)x 
  + \left(k_-(z)+i\sqrt{\Re z}\right)y
  \,.
\]
Writing a Taylor expansion for the two real-valued functions
\[
  [0,1]\ni t\longmapsto\Re e^{t\Lambda_z(x,y)}
  \qquad\mbox{and}\qquad
  [0,1]\ni t\longmapsto\Im e^{t\Lambda_z(x,y)}
  \,,
\]
we obtain that, for some $t_1,t_2\in[0,1]$, 
\beq\label{Lambda}
  e^{\Lambda_z(x,y)}-1 
  = \Lambda_z(x,y) 
  + \frac 12\Big[\Re\big(\Lambda_z(x,y)^2e^{t_1\Lambda_z(x,y)}\big)
  +i\,\Im\big(\Lambda_z(x,y)^2e^{t_2\Lambda_z(x,y)}\big)\Big]\,.
\eeq
Notice that, for all $z\in\CS$, 
$x\geq0$ and $y\leq0$, 
$\Re \Lambda_z(x,y)\leq0$, 
hence 
\beq\label{LambdaBis}
  \frac 12\Big|\Re\big(\Lambda_z(x,y)^2e^{t_1\Lambda_z(x,y)}\big)
  +i\,\Im\big(\Lambda_z(x,y)^2e^{t_2\Lambda_z(x,y)}\big)\Big|\leq|\Lambda_z(x,y)|^2
  \,.
\eeq
%
Moreover, due to (\ref{expansion}), we have
\[
  \Lambda_z(x,y) = \frac{(\Im z-1)\,x+(\Im z +1)\,y}{2\sqrt{\Re z}} 
  + \frac{\beta_z\,x+\gamma_z\,y}{(\Re z)^{3/2}}\,,
\]
for some complex constants $\beta_z$ and $\gamma_z$ independent of $(x,y)$ 
and uniformly bounded with respect to~$z$.
As a consequence, (\ref{Lambda}) and (\ref{LambdaBis})~yield
\[
  e^{\Lambda_z(x,y)}-1 
  = \frac{1}{\sqrt{\Re z}}\left(\frac{(\Im z-1)\,x+(\Im z +1)\,y}{2} 
  + \delta_2(z ;x,y)\right) ,
\]
where, for all $z\in\CS$, $x\geq0$ and $y\leq0$, 
\[
 |\delta_2(z ;x,y)| \leq C_0\,\frac{1+x^2+y^2}{\sqrt{\Re z}}\,,
\]
with some $C_0>0$.
Summing up, (\ref{decompKinter})~reads
\beq\label{decompM}
  \mathcal{M}_z(x,y) = 
  |V|^{1/2}(x) \left(\tilde{\mathcal{M}}_z^0(x,y) + r_z(x,y)\right) V_{1/2}(y)
  \,,
\eeq
where ($x\geq0$, $y\leq0$)
\begin{align}
  \tilde{\mathcal{M}}_z^0(x,y)
  := \ & \frac{1}{2} \, e^{-i\sqrt{\Re z}\,(x+y)} \, 
  \big[(\Im z-1)\,x+(\Im z+1)\,y\big] 
  \nonumber\\
  = \ &  \frac{1}{2} \, e^{-i\sqrt{\Re z}\,(x+y)} \, 
  \big[\Im z \, (x+y)-(|x|+|y|)\big]
  \label{exprM0+-}
\end{align}
and
\beq\label{exprRem}
  r_z(x,y) := 
  e^{-i\sqrt{\Re z}\,(x+y)} \, \delta_2(z ;x,y) 
  + e^{-k_+(z)\,x+k_-(z)\,y} \, \delta_1(z)
\eeq
satisfies, with some positive constant $C$,
\beq\label{boundR}
  \forall z\in\CS\,, \ x\geq0\,, \ y\leq0 \,, \quad
  |r_z(x,y)|\leq \frac{C}{\sqrt{\Re z}} \, (1+x^2+y^2)
  \,.
\eeq

By a similar analysis, 
we get the decomposition of the form~(\ref{decompM}) 
for $x\leq0$ and $y\geq0$ as well, where ($x\leq0$, $y\geq0$)
\begin{align}
  \tilde{\mathcal{M}}_z^0(x,y) 
  := \ & \frac{1}{2}e^{-i\sqrt{\Re z}\,(x+y)}
  \big[(\Im z+1)\,x+(\Im z-1)\,y\big] 
  \nonumber\\
  = \ & \frac{1}{2}e^{-i\sqrt{\Re z}\,(x+y)}
  \big[\Im z \, (x+y)-(|x|+|y|)\big]
  \label{exprM0-+}
\end{align}
and the bound (\ref{boundR}) holds also for $x\leq0$, $y\geq0$.

The case $xy\geq0$ can also be treated alike,
by noticing that in this case the first term 
on the right-hand side of~(\ref{exprKernel}) satisfies
\[
 \left|\frac{1}{2k_\pm(z)} \, e^{-k_\pm(z)|x-y|}\right|\leq \frac{C'}{\sqrt{\Re z}}
\]
with some $C'>0$. 
Moreover, using~(\ref{expansion}),
\begin{multline*}
  \pm\frac{k_+(z)-k_-(z)}{2k_\pm(z)\big(k_+(z)+k_-(z)\big)}
  \, e^{-k_\pm(z)(|x|+|y|)}-\sqrt{\Re z}\,e^{-i\sqrt{\Re z}\,(x+y)} 
  \\
  = \frac{1}{2} \, e^{-i\sqrt{\Re z}\,(x+y)} \,
  \big[\Im z\,(x+y)-(|x|+|y|)\big] + \rho_z(x,y)
  \,,
\end{multline*}
where $\rho_z(x,y)$ satisfies the bound~(\ref{boundR}). 
The decomposition~(\ref{defM}) with~\eqref{def.m} is therefore proved.

To complete the proof of the lemma, 
it remains to prove the uniform boundedness of~$\mathcal{M}_z$.
This can be deduced from~(\ref{defM}) and~\eqref{def.m}.
Indeed, with some $C_1>0$, we have, for $\Re z\geq1$,
\begin{equation*}
  \|M_z\|_\mathrm{HS}^2 
  \leq C_1 \int_{\mathbb{R}^2}|V(x)| \, (1+x^2+y^2)^2 \, |V(y)| \,dx\,dy 
  \,,
\end{equation*}
where the right hand side is finite if $V\in L^1(\mathbb{R},(1+x^4)\,dx)$
and actually independent of~$z$.
If $\Re z\leq1$, 
then according to~(\ref{estBS1}) and the expression~(\ref{defL}) 
of the kernel $\mathcal{L}_z$, we have
\[
  \|M_z\|_\mathrm{HS}\leq \|K_z\|_\mathrm{HS} 
  + \|L_z\|_\mathrm{HS}\leq C_2 \,
  \sqrt{\int_{\mathbb{R}^2}|V(x)| \, (1+|x|+|y|)^2 \, |V(y)| \,dx\,dy}
\]
with some $C_2>0$, 
hence the norm $\|M_z\|_\mathrm{HS}$ is uniformly bounded for $\Re z\leq 1$ as well.
\end{proof}
\begin{Remark}
Using a first-order expansion in~(\ref{Lambda}) instead of the second-order expansion, 
we would obtain the uniform boundedness of~$M_z$ 
under the weaker assumption $V\in L^1\big(\mathbb{R},(1+x^2)\,dx)$. 
However, the second-order expansion in~(\ref{Lambda}) is required 
in order to get the exact expression~(\ref{exprM0+-}) 
of the principal term $\tilde{\mathcal{M}}_z^0(x,y)$ in~(\ref{decompM}).
\end{Remark}

Since~$\|M_z\|$ is uniformly bounded with respect to $z \in \CS$, 
the operator $(1+\varepsilon M_z)$ is boundedly invertible 
for all~$\varepsilon$ small enough.
Consequently, in view of the identity 
\[
  \varepsilon K_z+1 
  = \varepsilon (L_z+M_z)+1 
  = (1+\varepsilon M_z)\big[\varepsilon(1+\varepsilon M_z)^{-1} L_z+1\big]
\]
and Theorem~\ref{thmBirmanSchwinger}, we have (for all $z \in \CS$)
\beq\label{equivspec}
  z\in\sigma_\mathrm{p}(H_\eps)
  \quad\Longleftrightarrow\quad 
  -1\in\sigma\big(\varepsilon(1+\varepsilon M_z)^{-1} L_z\big)
  \,.
\eeq

From the definition~(\ref{defL}) we see that~$L_z$ is a rank-one operator.
Consequently, $\varepsilon(1+\varepsilon M_z)^{-1} L_z$ is of rank one too.  
Indeed, for all $f\in L^2(\mathbb{R})$, we have
\[
  \varepsilon(1+\varepsilon M_z)^{-1} L_z f 
  = \varepsilon\sqrt{\Re z} \, (f,\bar\psi_z) \, (1+\varepsilon M_z)^{-1}\phi_z
  \,,
\]
where
\[  
  \phi_z(x) := e^{-i\sqrt{\Re z}\,x} \, |V|^{1/2}(x)
  \qquad\mbox{and}\qquad
  \psi_z(x) := e^{-i\sqrt{\Re z}\,x} \, V_{1/2}(x)
  \,.
\]
It follows that $\varepsilon(1+\varepsilon M_z)^{-1} L_z$
has the unique non-zero eigenvalue
\[
  \varepsilon\sqrt{\Re z} \, 
  \left((1+\varepsilon M_z)^{-1}\phi_z,\bar\psi_z\right)
  .
\]
Equivalence~(\ref{equivspec}) thus reads
\beq\label{equivspec2}
  z\in\sigma_\mathrm{p}(H_\eps)
  \quad\Longleftrightarrow\quad 
  -1 = \varepsilon\sqrt{\Re z} \, 
  \left((1+\varepsilon M_z)^{-1}\phi_z,\bar\psi_z\right)
  \,.
\eeq
Note that the right hand side represents 
an implicit equation for~$z$.

Writing
\[
  (1+\varepsilon M_z)^{-1} 
  = \sum_{j=0}^{n-1}(-1)^j\varepsilon^j M_z^j 
  + (-1)^{n}\varepsilon^{n} M_z^n(1+\varepsilon M_z)^{-1}
  \,,
\]
the condition on the right hand side of~(\ref{equivspec2}) reads
\beq\label{EigenEquation}
  \frac{1}{\sqrt{\Re z}} 
  = \sum_{j=1}^n(-1)^j \left(M_z^{j-1}\phi_z,\bar\psi_z\right) 
  \varepsilon^j 
  + (-1)^{n+1}
  \left(M_z^n(1+\varepsilon M_z)^{-1}\phi_z,\bar\psi_z\right)
  \varepsilon^{n+1}
  \,.
\eeq
In the following we estimate each term 
on the right hand side of~(\ref{EigenEquation}).

For $j=1,\dots,n\,$, denoting
$$
 V^{\otimes j}(x_1,\dots,x_j) := V(x_1)\dots V(x_j)\,,
$$
and using the decomposition \eqref{decompKinter} with \eqref{exprM0-+}, 
we have 
\begin{align}
  \left(M_z^{j-1}\phi_z,\bar\psi_z\right) 
  =\ & \int_{\mathbb{R}^j}\mathcal{M}_z(x_1,x_2)\dots\mathcal{M}_z(x_{j-1},x_j)
  \,\phi_z(x_j)\,\psi_z(x_1) 
  \, dx_1\dots dx_j 
  \nonumber \\
  =\ & \int_{\mathbb{R}^j}
  \left(
  \prod_{\ell = 1}^{j-1}
  |V|^{1/2}(x_\ell)\big[\tilde{\mathcal{M}}_z^0(x_\ell,x_{\ell+1})
  +r_z(x_\ell,x_{\ell+1})\big]V_{1/2}(x_{\ell+1})
  \right) 
  \nonumber \\
  & \times |V|^{1/2}(x_j) \, e^{-i\sqrt{\Re z}\,(x_1+x_j)} 
  \, V_{1/2}(x_1) \, dx_1\dots dx_j 
  \nonumber \\
  =\ & \int_{\mathbb{R}^j}e^{-i\sqrt{\Re z}\,(x_1+x_j)} \, V^{\otimes j}(x_1,\dots,x_j) 
  \nonumber \\
  & \times \prod_{\ell=1}^{j-1}
  \big[\tilde{\mathcal{M}}_z^0(x_\ell,x_{\ell+1})+r_z(x_\ell,x_{\ell+1})\big] 
  dx_1\dots dx_j 
 \nonumber \\
  =\ & I_{j-1}(z)+R_{j-1}(z)
  \,, \label{decompI+R}
\end{align}
where
\begin{multline}
  I_{j-1}(z)  := 
  \frac{1}{2^{j-1}}\int_{\mathbb{R}^j}e^{-2i\sqrt{\Re z} \, \sum_{\ell=1}^jx_\ell} \,
  V^{\otimes j}(x_1,\dots,x_j) 
  \\
  \times\prod_{\ell = 1}^{j-1} 
  \big[\Im z\,(x_\ell+x_{\ell+1})-(|x_\ell|+|x_{\ell+1}|)\big]
  dx_1\dots dx_j \label{exprIj}
\end{multline}
and  
$
 R_{j-1}(z) := (M_z^{j-1}\phi_z,\bar\psi_z)-I_{j-1}(z)
$
contains all the integral terms involving at least one factor 
of the form $r_z(x_\ell,x_{\ell+1})$. 
Using~(\ref{boundR}), one can easily check that
\beq\label{estRj}
 R_{j-1}(z) = \mathcal{O}\left(\frac{1}{\sqrt{\Re z}}\right)
\eeq
whenever $V\in L^1(\mathbb{R}, (1+x^{2n})\,dx)$.

On the other hand, we have
\[
  \prod_{\ell = 1}^{j-1}
  \big[\Im z \, (x_\ell+x_{\ell+1})-(|x_\ell|+|x_{\ell+1}|)\big] 
  = \sum_{\vec\ell\in \mathcal{J}_{j-1}}
  \prod_{m=1}^{j-1}(\Im z\,x_{\ell_m}-|x_{\ell_m}|)
  \,,
\]
for a subset $\mathcal{J}_{j-1}\subset \{1,\dots,j\}^{j-1}$ such that, 
for all $\vec\ell\in\mathcal{J}_{j-1}$, 
each coordinate in $\vec\ell$ is repeated at most twice. 
Consequently, separating the variables in (\ref{exprIj}), 
we get, for some positive integer $M_j$, 
\beq\label{exprIFourier}
 I_{j-1}(z) = \frac{1}{2^{j-1}}\sum_{k=1}^{M_j}I_{j-1}^{(k)}(z)\,,
\eeq
where each term $I_{j-1}^{(k)}(z)$ has the form
\begin{eqnarray*}
  I_{j-1}^{(k)}(z) &=& 
  \left(\int_\mathbb{R}e^{-2i\sqrt{\Re z}\,x} \, V(x) \, dx\right)^{a_{k,j}}
  \\
  &&\times 
  \left(\int_\mathbb{R}e^{-2i\sqrt{\Re z}\,x} (\Im z\,x-|x|) \, V(x) \, dx
  \right)^{b_{k,j}}  
  \\
  &&\times
  \left(\int_\mathbb{R}e^{-2i\sqrt{\Re z}\,x} \, 
  (\Im z\,x-|x|)^2V(x) \, dx\right)^{c_{k,j}}
  \,,
\end{eqnarray*}
with $a_{k,j},b_{k,j},c_{k,j}$ such that
\[
 \left\{
 \begin{array}{l}
  a_{k,j}>0 \,, \ b_{k,j}\geq0 \,, \ c_{k,j}\geq0\,,\\
  a_{k,j}+b_{k,j}+c_{k,j}=j\,,\\
  b_{k,j}+2c_{k,j}=j-1\,.
 \end{array}
\right.
 \]
Thus, 
if $\mathcal{F}[f](\xi)$ denotes the Fourier transform of~$f$ at point~$\xi$, 
we have
\begin{eqnarray}
 I_{j-1}^{(k)}(z) &=&\left(\mathcal{F}[V](2\sqrt{\Re z})\right)^{a_{k,j}}
  \left(\mathcal{F}\big[(\Im z\,x-|x|)V(x)\big](2\sqrt{\Re z})\right)^{b_{k,j}} 
  \nonumber\\
  &&\times\left(
  \mathcal{F}\big[(\Im z\,x-|x|)^2V(x)\big](2\sqrt{\Re z})
  \right)^{c_{k,j}}
 \,. \label{exprIFourierBis}
\end{eqnarray}

Now, since for $s=1,2$ the function $x\mapsto (\Im z\,x-|x|)^sV(x)$
belongs to $L^1(\mathbb{R})$ by assumption, 
its Fourier transform is in $L^\infty(\mathbb{R})$ and it is continuous. 
Hence there exists $M_1>0$ such that, for all $z\in\CS$ and $s=1,2$,
\[
  \left|\mathcal{F}\big[(\Im z\,x-|x|)^s V(x)\big](2\sqrt{\Re z})\right|
  \leq M_1
  \,.
\]
Similarly, since $V\in W^{1,1}(\mathbb{R})$, 
the function $\xi\mapsto\xi\,\mathcal{F}[V](\xi)$ 
belongs to $L^\infty(\mathbb{R})$ and it is continuous.
Hence there exists $M_2>0$ such that, for all $z\in\CS$, 
\[
  \left|\mathcal{F}[V](2\sqrt{\Re z})\right|\leq\frac{M_2}{\sqrt{\Re z}}
  \,.
\]
Thus~(\ref{exprIFourier}) and~(\ref{exprIFourierBis}) give
\beq\label{estIj}
 I_{j-1}(z) = \mathcal{O}\left(\frac{1}{\sqrt{\Re z}}\right)\,.
\eeq

Finally, (\ref{decompI+R}), (\ref{estRj}) and (\ref{estIj}) yield
\[
  \left(M_z^{j-1}\phi_z,\bar\psi_z\right)
  = \mathcal{O}\left(\frac{1}{\sqrt{\Re z}}\right)
\]
for all $j=1,\dots,n$.
Thus, according to~(\ref{EigenEquation}),
\[
  \frac{1}{\sqrt{\Re z}} \, \big(1-\mathcal{O}(\varepsilon)\big) 
  = (-1)^{n+1}
  \left(M_z^n(1+\varepsilon M_z)^{-1}\phi_z,\bar\psi_z\right)
  \varepsilon^{n+1}
  \,,
\]
uniformly with respect to $z$ as $\varepsilon\rightarrow0$.
We then notice that the right hand side in the above identity 
has the form $\mathcal{O}(\varepsilon^{n+1})$, 
uniformly with respect to $z$, as $\varepsilon\rightarrow0$. 
Therefore, we have
\[
  \frac{1}{\sqrt{\Re z}} = \mathcal{O}(\varepsilon^{n+1})
  \,,
\]
which concludes the proof of Theorem \ref{thmSpecBound2}.
\hfill\qed

\section{Examples}\label{Sec.Ex}
%

\subsection{Dirac interaction}
In order to test our results on an explicitly solvable model,
let us consider the operator formally given by the expression
$$
  H_\alpha = -\frac{d^2}{dx^2} + i \sgn(x) + \alpha \, \delta(x)
  \,, \qquad
  \alpha \in \Com
  \,,
$$
where~$\delta$ is the Dirac delta function.   
In fact, $H_\alpha$~can be rigorously defined 
(\cf~\cite[Ex.~5.27]{KS-book})
as the m-sectorial operator in $\sii(\Real)$
associated with the form sum $h+\alpha v$, where
$$
  v(\psi,\phi) := \psi(0)\bar{\phi}(0)
  \,, \qquad
  \Dom(v) := W^{1,2}(\Real)
  \,.
$$
We have
$$
\begin{aligned}
  (H_\alpha\psi)(x) &= -\psi''(x) + i \sgn(x)\,\psi(x) 
  \quad \mbox{for a.e. } x \in \Real
  \,, 
  \\
  \Dom(H_\alpha) &= 
  \left\{
  \psi \in W^{1,2}(\Real) \cap W^{2,2}(\Real\setminus\{0\})
  : \
  \psi'(0^+) - \psi'(0^-) = \alpha \, \psi(0)
  \right\}
  \,.
\end{aligned}
$$
It is also possible to show that~$H_\alpha$ is $\mathcal{T}$-self-adjoint.

Using for instance \cite[Corol.~IX.4.2]{Edmunds-Evans},
we have the stability result
$$
  \sigma_\mathrm{ess}(H_\alpha) = \sigma_\mathrm{ess}(H)
  = [0,+\infty) + i \, \{-1,+1\}
$$
for all $\alpha \in \Com$.
Since~$H_\alpha$ is $\mathcal{T}$-self-adjoint,
the residual spectrum of~$H_\alpha$ is empty (\cf~\cite[Sec.~5.2.5.4]{KS-book}).
Finally, the eigenvalue problem for~$H_\alpha$ 
can be solved explicitly and we find that $H_\alpha$~possesses
a unique (discrete) eigenvalue given by 
\begin{equation}\label{ev.delta}
  \lambda(\alpha) := \frac{1}{\alpha^2} - \frac{\alpha^2}{4} 
\end{equation}
if, and only if, 
\begin{equation}\label{condAlpha}
\lambda(\alpha) \not\in [0,+\infty) + i \, \{-1,+1\}
\,.
\end{equation}
In particular, the eigenvalue exists for all 
$\alpha \in \Real \setminus\{0\}$ and in this case it is real. 
It is interesting that the rate at which $\lambda(\alpha)$
tends to infinity as $\alpha \to 0$ coincides 
with the bound of Theorem~\ref{thmSpecBound1},
even if this theorem does not apply to the present singular potential
and even for non-real~$\alpha$.

Now, in order to state the condition~(\ref{condAlpha}) 
more explicitly in terms of $\alpha$, let us set, 
for all $\sigma=(\sigma_1,\sigma_2,\sigma_3)\in\{-1,+1\}^3$,
$$
 \Gamma_\sigma 
  :=
  \left\{ \sigma_1\sqrt{-2(r+i\sigma_2) + 2 \sigma_3\sqrt{r(r+2i\sigma_2)}}
  \, : \ r\in[0,+\infty)\right\}
 \,.
$$
Notice that, for all $r\in[0,+\infty)$, 
the square roots in the expression above are well defined.
Then the condition~\eqref{condAlpha} is equivalent to $\alpha\notin\Gamma$, 
where
\begin{equation}\label{defDomGamma}
\Gamma := \bigcup_{\sigma\in\{-1,+1\}^3}\Gamma_\sigma\,.
\end{equation}
The curve~$\Gamma$ is represented in Figure~\ref{figGamma}.

\begin{figure}[H]\begin{center}
\includegraphics[width=0.7\textwidth]{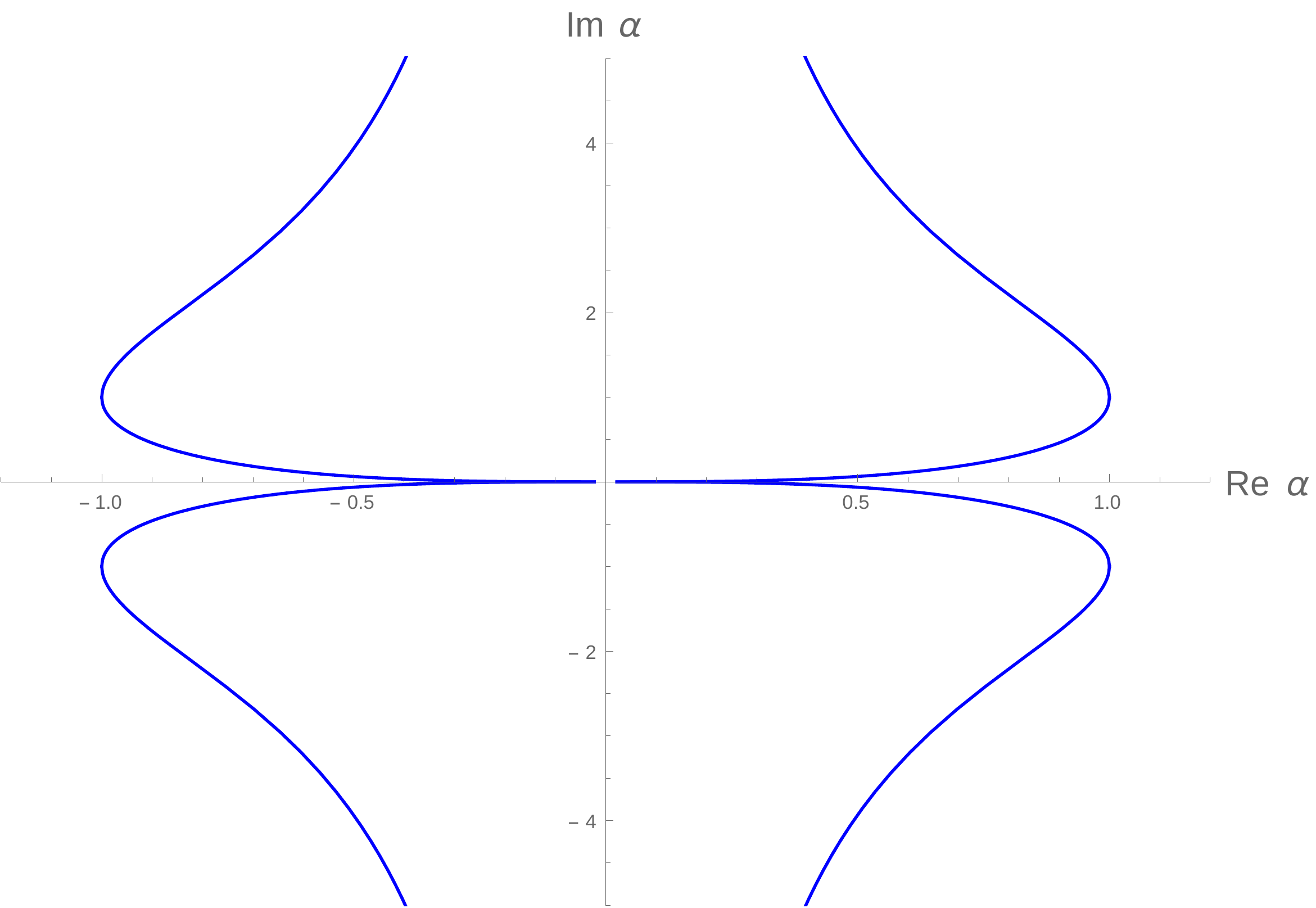}
\caption{{The curve~$\Gamma$ in the complex plane
representing values of~$\alpha$ 
for which the eigenvalue of~$H_\alpha$
does not exist.}}\label{figGamma}
\end{center}
\end{figure}

Let us summarise the spectral properties into the following proposition.
\begin{proposition}
For any $\alpha \in \Com$, we have
$$
\begin{aligned}
  \sigma_\mathrm{r}(H_\alpha) &= \varnothing \,,
  \\ 
  \sigma_\mathrm{c}(H_\alpha) &= [0,+\infty) + i \, \{-1,+1\} \,,	
  \\
  \sigma_\mathrm{p}(H_\alpha) &=
  \begin{cases}
    \varnothing & \mbox{if} \quad 
    \alpha \in \Gamma \,,
    \\
    \{\lambda(\alpha)\} & \mbox{if} \quad 
   \alpha \not\in \Gamma \,,
  \end{cases}
\end{aligned}
$$
where~$\lambda(\alpha)$ is given by~\eqref{ev.delta} 
and~$\Gamma$ is the domain defined in~\eqref{defDomGamma}.
\end{proposition}

\subsection{Step-like potential}
To have a definitive support for the existence of discrete spectra
for the operators of the type~\eqref{operator.weak}, 
here we consider $\eps=1$ and the following step-like profile 
for the perturbing potential:
$$
  V_{a,b}(x) := \left(- i \, \sgn(x) - b\right) \chi_{[-a,a]}(x) 
  \,,
$$
where $a>0$ and $b \in \Com$.
We set $H_{a,b} := H+V_{a,b}$.
By Proposition~\ref{propStabEss}, 
\begin{equation}\label{ess.step}
  \sigma_\mathrm{ess}(H_{a,b}) = [0,+\infty) + i \, \{-1,+1\}
\end{equation}
for all $a>0$ and $b \in \Com$.

The differential equation of the eigenvalue problem 
$H_{a,b}\psi=\lambda\psi$ can be solved in terms
of sines and cosines in each of the intervals
$(-\infty,-a)$, $(-a,a)$ and $(a,+\infty)$.
Choosing integrable solutions in the infinite intervals
and gluing the respective solutions at~$\pm a$
by requiring the $W^{2,2}$-regularity, 
we arrive at the following equation
\begin{equation}\label{implicit}
  \big[\sqrt{\lambda^2+1}-\lambda-b\big] 
  \frac{\sin\big(2 a \sqrt{\lambda+b}\big)}{\sqrt{\lambda+b}}
  -i \big(\sqrt{\lambda+i}-\sqrt{\lambda-i}\big) \, 
  \cos\big(2 a \sqrt{\lambda+b}\big)
  = 0
\end{equation}
for eigenvalues~$\lambda$ satisfying $|\Im\lambda| < 1$
and $\lambda+b \not\in (-\infty,0)$.
The equation for the case $\lambda=-b$ is recovered after taking
the limit $\lambda \to -b$ in the above equation.
For eigenvalues~$\lambda$ satisfying 
$|\Im\lambda| < 1$ and $\lambda+b\in(-\infty,0)$, we find
$$
\big[\sqrt{\lambda^2+1}-\lambda-b\big] 
  \frac{\sinh\big(2 a \sqrt{|\lambda+b|}\big)}{\sqrt{|\lambda+b|}}
  -i \big(\sqrt{\lambda+i}-\sqrt{\lambda-i}\big) \, 
  \cosh\big(2 a \sqrt{|\lambda+b|}\big)
  = 0
  \,.
$$
In the same manner, it is possible to derive equations
for eigenvalues~$\lambda$ satisfying $|\Im\lambda| \geq 1$.
However, we shall not present these formulae, 
for below we are particularly interested in real eigenvalues. 
We only mention that it is easy to verify that no point 
in the essential spectrum~\eqref{ess.step} can be an eigenvalue.

Henceforth, we investigate the existence of real eigenvalues.
Moreover, we restrict to real~$b$ and look for eigenvalues $\lambda > -b$,
so that it is enough to work with~\eqref{implicit}.
First of all, notice that, for any $\lambda>-b$ satisfying~\eqref{implicit}, 
$\sin\big(2 a \sqrt{\lambda+b}\big)$ never vanishes.
At the same time, $\Im\sqrt{\lambda+i}$ is non-zero for real~$\lambda$.
We can thus rewrite~\eqref{implicit} as follows
$$
  \cot\big(2 a \sqrt{\lambda+b}\big) 
  = -\frac{\sqrt{\lambda^2+1}-(\lambda+b)}{2 \, \sqrt{\lambda+b} \ \Im\sqrt{\lambda+i}}
  \sim b 
  \qquad \mbox{as} \quad
  \lambda \to + \infty \,.
$$ 
Since there is a periodic function with range~$\Real$
on the left hand side, it follows from the asymptotics that~$H_{a,b}$
possesses infinitely many eigenvalues for every real~$b$.
Let us highlight this result by the following proposition.
\begin{proposition}
For any $a>0$ and $b \in \Real$, 
$H_{a,b}$ possesses infinitely many distinct real discrete eigenvalues. 
\end{proposition}

Several real eigenvalues of~$H_{a,b}$ 
as functions of $b \in \Real$ are represented in Figure~\ref{figev}.

\begin{figure}[h]\begin{center}
\includegraphics[width=0.7\textwidth]{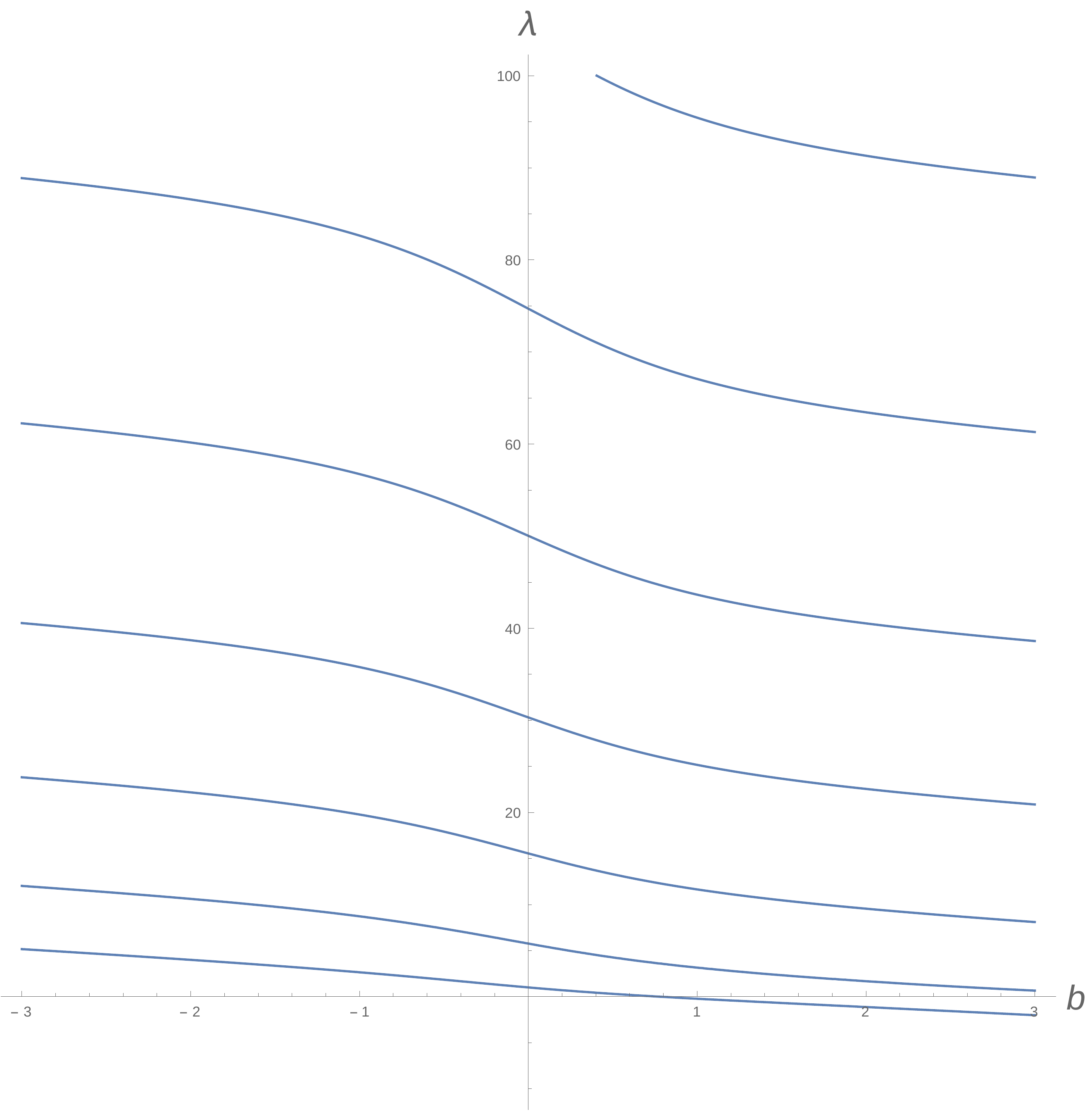}
\caption{{Dependence of real eigenvalues of~$H_{a,b}$ 
on~$b$ for $a=1$.}}\label{figev}
\end{center}
\end{figure}

\subsection{Dirichlet realisation}\label{Sec.Dirichlet}
Since the spectrum of~$H$ is the union of the two half-lines 
$\mathbb{R}_++i$ and $\mathbb{R}_+-i$, 
one might have expected the operator~$H$ to behave as some sort of 
decoupling of two operators $-\frac{d^2}{dx^2} + i$ in $\sii(\Real_+)$
and $-\frac{d^2}{dx^2}-i$ in $\sii(\Real_-)$.
The existence of non-trivial pseudospectra 
(\cf~Theorem~\ref{thmEstResIntro}) actually indicates 
that this is not the case.
Indeed, the situation strongly depends 
on the way the operator is defined, emphasising the importance of the choice
of domain in the pseudospectral behaviour of an operator.

For comparison,
let~$H^D$ be the operator in $\sii(\Real)$ 
that acts as~$H$ in $\Real_+^*:=(0,+\infty)$ 
and $\Real_-^*:=(-\infty,0)$,
but satisfies an extra Dirichlet condition at zero, \ie,  
\[
 \Dom(H^D) := \big(W^{2,2}\cap W_0^{1,2}\big)\big(\mathbb{R}\setminus\{0\}\big)\,.
\]
Considering this operator instead of~$H$ means 
that the previous matching conditions at $x=0$,
$u(0^-) = u(0^+)$ and $u'(0^-) = u'(0^+)$ for $u \in \Dom(H)$,
are replaced by the conditions $u(0^-) = 0 = u(0^+)$ for $u \in \Dom(H^D)$.

We can write~$H^D$ as a direct sum
\beq\label{decompA0}
  H^D = H^D_- \oplus H^D_+
  \,,
\eeq
where $H^D_\pm$ are operators in $L^2(\mathbb{R}_\pm^*)$ defined by
\beq\label{defA0+-}
  H^D_\pm := -\frac{d^2}{dx^2} \pm i 
  \,, \qquad
  \Dom(H^D_\pm) := \big(W^{2,2}\cap W_0^{1,2}\big)(\mathbb{R}_\pm^*)
  \,.
\eeq
Since the spectra of~$H^D_\pm$ are trivially found,
we therefore have (see \cite[Sec.~IX.5]{Edmunds-Evans})
\[
  \sigma(H^D) 
  = \sigma(H^D_-)\cup\sigma(H^D_+) = \mathbb{R}_++i\,\{-1,+1\}
  \,.
\]
Hence~$H^D$ and~$H$ have the same spectrum
(\cf~Proposition~\ref{propSpectrum}).

We can also decompose the resolvent of~$H_D$ as follows
\[
 (H^D-z)^{-1} = (H^D_--z)^{-1}\oplus(H^D_+-z)^{-1}
\]
for every $z \not\in \mathbb{R}_++i\,\{-1,+1\}$.
Since~$H^D_\pm$ are obtained from self-adjoint operators shifted by a constant, 
they both have trivial pseudospectra. 
Consequently, $H^D$~has trivial pseudospectra as well. 
In other words, although~$H^D$ and~$H$ have the same spectrum, 
that of~$H$ is far more unstable (\cf~Theorem~\ref{thmEstResIntro}).

To be more specific, let us write down 
the integral kernel $\mathcal{R}_z^D$ of $(H^D-z)^{-1}$.
For $f\in L^2(\mathbb{R})$, 
the function $(H^D-z)^{-1}f$ has the form~(\ref{exprSol}), 
where the constants $A_+,A_-,B_+,B_-$ are uniquely determined 
by the Dirichlet condition at~$0$ together with the condition 
$(H^D-z)^{-1}f(x) \to 0$ as $x\to\pm\infty$. 
The former yields $B_+ = -A_+$ and $B_- = -A_-$, 
while the latter gives the following values for $A_+$ and $A_-$:
\[
  A_+ = \frac{1}{2k_+(z)}\int_0^{+\infty}f(y) \, e^{-k_+(z)y} \, dy
  \,, \qquad
  A_- = -\frac{1}{2k_-(z)}\int_{-\infty}^0f(y) \, e^{k_+(z)y} \, dy
  \,.
\]
Eventually, we obtain the following expression for the integral kernel:
\[
  \mathcal{R}_z^D(x,y) 
  = \frac{1}{2k_\pm(z)}\left(
  e^{-k_\pm(z)|x-y|}-e^{-k_\pm(z)(|x|+|y|)}
  \right)\chi_{\Real_\pm}(y)
  \,, \qquad \pm x>0
  \,.
\]

Now, as in Section~\ref{ssDefA+V},
we can consider the perturbed operator 
$$
  H_\eps^D := H^D \dot{+} \eps V
$$
for any $V\in L^1(\mathbb{R})$. 
We claim that,
under the additional assumption $V\in L^1(\mathbb{R},(1+x^2)\,dx)$,
the Hilbert-Schmidt norm of the Birman-Schwinger operator
\[
  K_z^D := |V|^{1/2} \, (H^D-z)^{-1} \, V_{1/2} 
\]
is uniformly bounded with respect to $z\notin\mathbb{R}^++i\{-1,1\}$. 
To see it, let us first assume $x>0$. 
If $|z-i|\leq c_0$ for some positive $c_0$, then
\begin{align*}
  |\mathcal{R}_z^D(x,y)| 
  & \leq 
  \frac{1}{2|k_+(z)|}\left(\big|e^{-k_+(z)|x-y|}-1\big|
  +\big|(e^{-k_+(z)(|x|+|y|)}-1\big|\right) 
  \\
  & \leq \frac{|x-y|+|x|+|y|}{2}
  \,,
\end{align*}
where we have used the inequality $|e^{-\omega}-1|\leq |\omega|$ for $\Re\omega\geq0$.
On the other hand, if $|z-i|>c_0$, 
then $|k_+(z)|$ is uniformly bounded from below, 
hence $\mathcal{R}_z^D(x,y)$ is uniformly bounded  
with respect to $x\geq0$, $y\in\mathbb{R}$ and~$z$ such that $|z-i|>c_0$.
The same analysis can be performed for $x<0$, 
thus there exists $C>0$ such that, 
for all $(x,y)\in\mathbb{R}^2$ and $z\notin[0,+\infty)+i\{-1,1\}$, 
\[
 |\mathcal{R}_z^D(x,y)|\leq C(1+|x|+|y|)\,.
\]
Consequently, the computation of the Hilbert-Schmidt norm of $K_z^D$ yields
\beq
  \|K_z^D\|_\mathrm{HS}
  \leq C \int_\mathbb{R}(1+x^2)|V(x)| \, dx
  \,.
\eeq

After noticing that 
$\sigma_\mathrm{ess}(H^D_\eps) = \sigma_\mathrm{ess}(H^D)$ 
for all $\eps\in\Real$
(by the same arguments as in the proof of Proposition~\ref{propStabEss}),
the Birman-Schwinger principle
(\ie~a version of Theorem~\ref{thmBirmanSchwinger} for~$H^D_\eps$)
leads to the following statement.
\begin{proposition}\label{Prop.Hardy}
Let $V\in L^1\left(\mathbb{R},(1+x^2)\,dx\right)$.
There exists a positive constant $\eps_0>0$ such that, 
for all $\eps \in (0,\eps_0)$, we have 
\[
  \sigma(H^D_\eps) = \sigma(H^D) = \mathbb{R}_++i\,\{-1,1\}
  \,.
\]
\end{proposition}

In other words, in the simpler situation of the operator~$H^D$, 
we are able to prove the absence of weakly coupled eigenvalues. 
Proposition~\ref{Prop.Hardy} can be considered as some sort
of ``Hardy inequality'' or ``absence of virtual bound state''
for the non-self-adjoint operator~$H^D$.
Let us also notice that a similar result 
has been established by Frank~\cite{Frank_2011}
in the case of Schr\"odinger operators with complex potentials
in three and higher dimensions.

\subsection*{Acknowledgment}
%
The authors are grateful to Mark Embree for his Figure~\ref{fig.Embree}
and to Petr Siegl for valuable suggestions.
The authors also thank the anonymous referee for helpful comments.
The research was partially supported
by the project RVO61389005 and the GACR grant No.\ 14-06818S.
The first author acknowledges the support of the ANR project NOSEVOL.
The second author also acknowledges the award from 
the \emph{Neuron fund for support of science},
Czech Republic, May 2014.

%
\addcontentsline{toc}{section}{References}
\bibliography{bib}
\bibliographystyle{amsplain}
\end{document}